\documentclass[letterpaper,12pt]{amsart}
\usepackage{amsmath,amssymb,amsxtra,amsthm, amstext, amscd,amsfonts,fancyhdr, hyperref}
\newcommand{\bA}{{\mathbb{A}}}

\newcommand{\bC}{{\mathbb{C}}}

\newcommand{\bF}{{\mathbb{F}}}

\newcommand{\bH}{{\mathbb{H}}}

\newcommand{\bN}{{\mathbb{N}}}

\newcommand{\bP}{{\mathbb{P}}}
\newcommand{\bQ}{{\mathbb{Q}}}
\newcommand{\bR}{{\mathbb{R}}}

\newcommand{\bZ}{{\mathbb{Z}}}

\newcommand{\Bw}{{\mathbf{w}}}
\newcommand{\Bx}{{\mathbf{x}}}
\newcommand{\By}{{\mathbf{y}}}
\newcommand{\Bz}{{\mathbf{z}}}

\newcommand{\BB}{\mathbf{B}}

  \newcommand{\B}{{\mathcal{B}}}
  \newcommand{\C}{{\mathcal{C}}}
  \newcommand{\D}{{\mathcal{D}}}
  \newcommand{\E}{{\mathcal{E}}}
  \newcommand{\F}{{\mathcal{F}}}
  \newcommand{\G}{{\mathcal{G}}}

  \newcommand{\M}{{\mathcal{M}}}

\renewcommand{\P}{{\mathcal{P}}}
  \newcommand{\Q}{{\mathcal{Q}}}
  \newcommand{\R}{{\mathcal{R}}}
\renewcommand{\S}{{\mathcal{S}}}

  \newcommand{\Y}{{\mathcal{Y}}}

\newcommand{\fm}{\mathfrak{m}}
\newcommand{\fd}{\mathfrak{d}}
\newcommand{\Gal}{\operatorname{Gal}}

\newcommand{\sing}{\operatorname{sing}}

\newcommand{\ff}{\mathfrak{f}}
\newcommand{\fg}{\mathfrak{g}}

\newcommand{\ep}{\varepsilon}

\newcommand{\bbeta}{\boldsymbol \beta}
\newcommand{\balpha}{\boldsymbol \alpha}

\newcommand{\ol}{\overline}

\renewcommand{\phi}{\varphi}
\newcommand{\upchi}{{\raise.35ex\hbox{$\chi$}}}

\makeatletter
\@namedef{subjclassname@2010}{%
  \textup{2010} Mathematics Subject Classification}
\makeatother

\newtheorem{theorem}{Theorem}[section]

\newtheorem{proposition}[theorem]{Proposition}
\newtheorem{lemma}[theorem]{Lemma}

\theoremstyle{definition}
\newtheorem{definition}[theorem]{Definition}

\newtheorem{remark}[theorem]{Remark}

\numberwithin{equation}{section}

\begin{document}

\title{Power-free values of binary forms and the global determinant method}
\author{Stanley Yao Xiao}
\address{University of Waterloo, Dept.~of Pure Mathematics, Waterloo, ON, N2L 3G1, Canada}
\email{y28xiao@uwaterloo.ca}
\date{\today}

%%%%%%%%%%%%%%%%%%%%%%%%%%%%%%%%%%%%%%%%%%%%%%%%%%%%%%%%%%%%%%%%%%%

\begin{abstract} We give an improved estimate for the density of $k$-free values of integral binary forms with no fixed $k$-th power divisor. Further, we give the corresponding improvement to a theorem of Stewart and Top on the number of power-free values in an interval that may be assumed by a binary form. The approach we use involves a generalization of the global determinant method of Salberger.
\end{abstract}

\maketitle

\newpage

%%%%%%%%%%%%%%%%%%%%%%%%%%%%%%%%%%%%%%%%%
\section{Introduction}
\label{S1}
%%%%%%%%%%%%%%%%%%%%%%%%%%%%%%%%%%%%%%%%%

Let $F(x,y)$ be a binary form with integer coefficients, non-zero discriminant, and degree $D \geq 3$, such that the largest degree of an irreducible factor $f$ of $F$ over $\bQ$ is $d$. We say that an integer $n$ is \emph{$k$-free} if, for all primes $p$ dividing $n$, we have $p^k \nmid n$. In general, when $k \geq 2$, we expect that for a positive proportion of integer tuples $(x,y)$, that $F(x,y)$ is $k$-free; unless there is a reason for it not to be $k$-free. \\ \\
For any set $\S$, we denote by $\# \S$ the cardinality of $\S$. Write 
\begin{equation} \label{E0} \displaystyle \rho_F(m) = \# \{(i,j) \in \{0, \cdots, m-1\}^2 : F(i,j) \equiv 0 \pmod{m}\} \end{equation}
and 
\begin{equation}\label{E1}C_{F,k} = \prod_p \left(1 - \frac{\rho_F(p^k)}{p^{2k}}\right).\end{equation}
As we will show in Section \ref{S8}, and was shown by Filaseta in \cite{<Fil>}, the quantity $\rho_F(p^k) \ll p^{2k-2}$, whence the product in (\ref{E1}) converges absolutely since $k \geq 2$. Further, write 
\[N_{F,k}(B) =  \#\{ (x,y) \in \bZ^2 \cap [1,B]^2 :  F(x,y) \text{ is } k\text{-}\text{free}\}.\]
Suppose that there is no prime $p$ for which $p^k$ divides $F(x,y)$ for all $(x,y) \in \bZ^2$. In 1992, Greaves \cite{<Gre>} showed that as $(x,y)$ takes on values in $[1,B]^2 \cap \bZ^2$, the binary form $F(x,y)$ as above takes on, asymptotically as $B$ tends to $\infty$, $C_{F,k} B^2$ $k$-free values whenever $k \geq (d - 1)/2$. Filaseta improved this for irreducible binary forms (in which case $D = d$ in the above notation) to $k \geq (2\sqrt{2} - 1)d/4$ in \cite{<Fil>}. Hooley, in 2009, showed in \cite{<Hoo3>} that it suffices to take $k \geq (d-2)/2$. This improvement is significant for small degrees. In particular, it shows that suitable forms of degree 8 take on infinitely many cube-free values, a result unavailable until Hooley's paper. In 2011, Browning \cite{<B2>} was able to apply the so-called determinant method to obtain that irreducible binary forms satisfying the necessary non-degenerate conditions are $k$-free as soon as $\displaystyle k > 7d/16$. The determinant method was pioneered by Bombieri and Pila in \cite{<BP>} and greatly extended by Heath-Brown in \cite{<HB1>} and again by Salberger in \cite{<S1>} and \cite{<S2>}. The key to Browning's improvement is the so-called global determinant method introduced by Salberger in \cite{<S2>}. \\ \\
Granville showed, subject to the $abc$-conjecture, that appropriate binary forms $F(x,y)$ take on infinitely many square-free values in \cite{<Gran>}. Poonen showed in \cite{<Poo>} that general, not necessarily homogeneous, binary polynomials $F(x,y)$ with integer coefficients take on infinitely many square-free values assuming the $abc$-conjecture. However, one notes that Poonen's result does not lead to an asymptotic formula in general. \\ \\
For a real number $t$, let $\lceil t \rceil$ denote the least integer $u$ such that $t \leq u$. We obtain the following theorem: \\
\begin{theorem}\label{MT1} Let $F(x,y)$ be a binary form with non-zero discriminant of degree $D \geq 2$ with integer coefficients. Let $k \geq 2$ be an integer. Suppose that for each prime $p$, there exists a pair of integers $(x_0, y_0)$ such that $p^k$ does not divide $F(x_0, y_0)$. Let $d$ denote the largest degree of a factor $f$ of $F$ over $\bQ$. Then whenever 
\begin{equation} \label{main PF condition} k > \min \left \{\frac{7 d}{18}, \left \lceil \frac{d}{2} \right \rceil - 2 \right \},\end{equation}
we have
\begin{equation} \label{1E3} N_{F,k}(B) = C_{F,k}B^2 + O\left(\frac{B^2}{\log^\delta B}\right), \end{equation}
where $\delta = 0.7043$ if $k = 2, d = 6$ and $\delta = 1$ otherwise.
\end{theorem} 
For example, we have that $F(x,y)$ takes on infinitely many $6$-free values for $d \leq 15$. \\ \\
The value of $\delta$ in Theorem \ref{MT1} for the case $k = 2, d = 6$ is due to Helfgott \cite{<Hel>}. He obtained a better error term in (\ref{1E3}) for the cases $k = 2$ and $d= 3, 4, 5$ as well; see page 2 of \cite{<Hel>}. The condition $k > 7d/18$ in (\ref{main PF condition}) arises from the application of the global determinant method, and represents the main contribution of this paper. The condition $k > \lceil d/2 \rceil - 2$ is equivalent to the condition $d \leq 2k + 1$, which is exactly the condition required for Greaves' theorem in \cite{<Gre>}. This result is superior for small degrees. \\ \\
Mazur and Gouv\^{e}a showed in \cite{<GM>} that the problem of counting square-free values of binary forms can be applied to construct elliptic curves $E$ that possess many quadratic twists with large rank. They adapted methods introduced by Hooley in \cite{<Hoo>} to the context of binary forms. They remarked in \cite{<GM>} that the sieve method developed by G. Greaves in \cite{<Gre>} is more efficient at counting square-free values of binary forms and can be used to strengthen their result. Stewart and Top, in \cite{<ST>}, were able to achieve this. In particular, they proved as Theorem 1 in \cite{<ST>} that  for $F(x,y)$ a binary form with integral coefficients of degree $D \geq 3$ and non-zero discriminant, there exists a positive constant $C$ for which $F$ assumes at least $C B^{2/D}$ $k$-free values in the interval $[-B, B]$, provided that $k \geq (d-1)/2$ or if $k = 2, d \leq 6$. The condition $k \geq (d-1)/2$ or if $k = 2, d \leq 6$ corresponds precisely to the theorem of Greaves in \cite{<Gre>}. The argument used to prove Theorem 1 \cite{<ST>} is mostly independent of the arguments used in Greaves \cite{<Gre>}, whence we can improve Theorem 1 in \cite{<ST>} by providing a better estimate for $k$-free values of binary forms.  Analogous to \cite{<ST>}, we define the counting function $R_{F,k}(B)$ as follows:
$$\displaystyle R_{F,k}(B) = \# \{t \in \bZ: |t| \leq B, \exists (x,y) \in \bZ^2 \text{ such that } F(x,y) = t, t \text{ is } k\text{-free} \}.$$
We then have the following result:
\begin{theorem} \label{MT2} Let $k \geq 2$. Let $F(x,y)$ be a binary form of degree $D \geq 3$ with integer coefficients and non-zero discriminant, with no fixed $k$-th power prime divisor. Let $d$ be the largest degree of an irreducible factor of $F$ over $\bQ$ and suppose that 
$$ \displaystyle k > \min \left\{\frac{7d}{18}, \left \lceil \frac{d}{2} \right \rceil - 2 \right \}.$$ 
Then there exist positive real numbers $C_1$ and $C_2$, which depend on $F$ and $k$, such that if $B > C_1$, then
$$\displaystyle R_{F,k}(B) > C_2 B^{2/D}.$$
\end{theorem}
There is an analogous question for polynomials of a single variable. Suppose that $g(x)$ is a polynomial with integer coefficients and degree $d$ which is irreducible over $\bQ$ and has no fixed $k$-th power prime divisor. Then we expect that $g(x)$ should take on infinitely many $k$-free values for $k \geq 2$. Indeed, this was established conditionally assuming the $abc$-conjecture by Granville \cite{<Gran>}; see also \cite{<MP>}. For larger values of $k$, the investigation goes back to Ricci in 1933 \cite{<Ric>}, who established that $g$ takes on infinitely many $k$-free values for $k \geq d$. Erd\H{o}s \cite{<E>}, in 1956, showed that $k \geq d-1$ suffices. However, Erd\H{o}s only gave a lower bound and not an asymptotic formula. Hooley was able to obtain the exact asymptotic formula in terms of local densities in 1967 \cite{<Hoo>}. This point will be elaborated below. \\ \\
For each positive integer $m$, define $\rho_g(m)$ to be the cardinality of the set $\{i \in \{0, \cdots, m-1\} : g(i) \equiv 0 \pmod{m}\}$. Put \begin{equation}\label{E2} \displaystyle c_{g,k} = \prod_p \left(1 - \frac{\rho_g(p^k)}{p^k} \right), \end{equation}
which is well defined (that is, the product converges) when $k \geq 2$. It is non-zero precisely when $g$ does not have a fixed $k$-th power prime divisor. Write
$$N_{g,k}(B) = \# \{1 \leq x \leq B : g(x) \text{ is } k\text{-free}\}. $$
Then, one should expect that
\begin{equation} \label{E3}\displaystyle N_{g,k}(B) \sim c_{g,k}B. \end{equation}
Indeed, this was the result obtained by Hooley, under the assumption that $k \geq d-1$. A similar asymptotic formula was obtained by all subsequent authors. Nair obtained (\ref{E3}) under the assumption $k \geq \left(\sqrt{2} - \frac{1}{2} \right)d$ in 1976 \cite{<Nai1>}. Heath-Brown obtained (\ref{E3}) under the assumption that $k \geq (3d+2)/4$ in 2006 \cite{<HB2>}, where he used the determinant method. Browning improved Heath-Brown's result to $k \geq (3d+1)/4$ in \cite{<B2>}. We will give another proof of Browning's result in Section \ref{S10} as an illustration of our method.\\ \\
It should be noted that Heath-Brown obtained (\ref{E3}) for irreducible polynomials of the shape $f(x) = x^d + c, c \in \bZ$ assuming $k \geq (5d+3)/9$ in \cite{<HB4>}. His arguments are also inspired by weighted projective spaces, defined below, but are materially different from the arguments presented in the present paper. It would be interesting to see whether Theorem \ref{MT1} can be improved for diagonal forms of the shape $F(x,y) = Ax^d + By^d$. \\ \\
%Heath-Brown's generalization of the determinant method in \cite{HB1} was a significant advancement in the theory of quantitative arithmetic on algebraic varieties. He obtained a flexible and general method to bound from above the number of rational points of bounded height on projective hypersurfaces. His method was dubbed the `$p$-adic determinant method' by various authors (see: \cite{S1}, \cite{B1}, \cite{Bro}, for example). The name comes from the estimation of a certain determinant $\Delta$ of a matrix $\M$ where the columns are indexed by a set of monomials and the rows are indexed by primitive integral representations of rational points of bounded height on the projective surface. Since $\Delta$ consists of integer entries which are bounded, it suffices to show that $\Delta$ is divisible by a very large positive integer to conclude that $\Delta = 0$. This is the key difference between Heath-Brown's approach and the earlier approach of Bombieri and Pila in \cite{BP}. Instead of trying to obtain an archimedean upper bound for the determinant $\Delta$, Heath-Brown looked for $p$-adic estimates. This work was extended by Salberger in \cite{S1}, who obtained improvements of Heath-Brown's theorem using algebraic geometry. \\ \\
In order to prove Theorem \ref{MT1} and Theorem \ref{MT2}, we generalize the $p$-adic determinant method of Heath-Brown, as extended by Salberger, to the case of \emph{weighted} projective spaces. Broberg had taken this perspective to study rational points on curves in the weighted projective plane in \cite{<Bro0>}. Recall that a projective space $\bP_\bF^{r+1}$ over a field $\bF$ is defined as the set of equivalence classes of $\bF^{r+2} \setminus \{\textbf{0}\}$ under the equivalence relation defined by 
\[\Bx = (x_0, \cdots, x_{r+1}) \sim \By = (y_0, \cdots, y_{r+1}) \]
if and only if there exists $\lambda \in \bF \setminus \{0\}$ such that
\[(x_0, \cdots, x_{r+1}) = (\lambda y_0, \cdots, \lambda y_{r+1}). \] 
Let $\mathbf{w} = (w_0, \cdots, w_{r+1})$ be a vector of positive integers, which we will call the \emph{weight vector}. The coordinates of the weight vector are called \emph{weights}. With a given weight vector $\mathbf{w}$, we can define the \emph{weighted projective space} $\bP_\bF(w_0, \cdots, w_{r+1})$ to be the set of equivalence classes of $\bF^{r+2} \setminus \{\textbf{0}\}$ under the equivalence relation 
\[\Bx \sim \By\]
if and only if there exists $\lambda \in \ol{\bF} \setminus \{0\}$, where $\ol{\bF}$ denotes an algebraic closure of $\bF$, such that
\[(x_0, \cdots, x_{r+1}) = (\lambda^{w_0} x_0, \cdots, \lambda^{w_{r+1}} x_{r+1}).\] 
%Heath-Brown, in his proof of Theorem 14 in \cite{HB1}, took a small set of auxiliary primes and considered them one at a time. Salberger in \cite{S2} devised a version of the determinant method that considered a set of primes inherent to the hypersurface at hand, and he then dealt with all of these primes simultaneously. This allows one to save on the total degree of the auxiliary form by a factor of $r$ when compared to Heath-Brown's theorem. The trade-off is that Heath-Brown's approach allows us to show that there are in fact many auxiliary forms of small degree which multiply to form a form of large degree, whereas Salberger's approach may produce a single irreducible auxiliary form of large degree. This distinction is delicate and will be made explicit in our Theorem \ref{MT3}. Further, the difference is relevant in our application. The details of the application of these results to the binary form problem can be found in Section \ref{S9}. \\ \\
Our Theorem \ref{MT3} generalizes Salberger's Theorem 2.2 in \cite{<S2>} and Heath-Brown's Theorem 14 in \cite{<HB1>}. The generalization of Heath-Brown and Salberger's determinant methods will form the technical heart of this paper.\\ \\
We then apply the determinant method mentioned above which applies to the weighted projective space setting to the weighted projective surface $X$ defined by the following equation:
\begin{equation} \label{BFE}  f(x,y) = vz^k,
\end{equation}
which is a surface in $\bP_\bQ(1,1,d-2k,2)$. Here $f$ is an irreducible factor of degree $d$ of the binary form $F$ given in Theorem \ref{MT1}. Applying the determinant method in this way allows us to deal with a dimension two subvariety $X$ inside the weighted projective space $\bP_\bQ(1,1,d-2k,2)$. This leads to a stronger result than we would obtain by dealing with a dimension three subvariety inside $\bA^4$ or working with a surface in $\bA^3$ by a priori fixing one variable, which was Browning's approach. We emphasize that viewing (\ref{BFE}) as a surface in weighted projective space is critical to our improvement.  \\ \\
We now make a remark regarding the choice of weights $(1,1,d-2k, 2)$. It seems a priori that the better weight choice is $(1,1,d-k, 1)$, which is similar to Heath-Brown's approach in \cite{<HB4>}. However the weight vector $(1,1,d-k,1)$ does not take into account the progress made by Greaves and will in fact produce results inferior to Greaves in \cite{<Gre>}. Nevertheless, in our proof of Theorem \ref{BThm} we will use $(1,1,d-k,1)$, precisely because Greaves' result does not apply in this context. \\ \\
Moreover, we remark that our approach does not seem to generalize in an obvious way to subsequent work by Browning, Heath-Brown, and Salberger dealing with arbitrary projective varieties in \cite{<BHS>}, because we do not know how to deal with projections of arbitrary weighted projective varieties onto a hypersurface in a weighted projective space of lower dimension. \\ \\
The outline of our paper is as follows. In Section \ref{S2}, we follow closely Salberger's argument in \cite{<S1>} to examine the Hilbert functions of weighted projective hypersurfaces. This allows us to extend some results found in \cite{<CLO>}. Our main result on the determinant method is Theorem \ref{MT3}, which is stated in Section \ref{S3}. The second part of Theorem \ref{MT3} is analogous to Salberger's Theorem 2.2 in \cite{<S2>}, and the first part is analogous to Heath-Brown's Theorem 14 in \cite{<HB1>}. We prove Theorem \ref{MT3} in Sections \ref{S5} and \ref{S6}. In Sections \ref{S8} and \ref{S9}, we follow the strategies of Heath-Brown and Salberger to apply the results in Sections \ref{S3} to \ref{S6} to prove Theorem \ref{MT1}. In Section \ref{S10}, we give another proof of Browning's theorem on $k$-free values of polynomials in \cite{<B2>} as an illustration of our approach. Finally, in Section \ref{S11}, we give a proof of Theorem \ref{MT2} which is a consequence of Theorem \ref{MT1} and the argument given in \cite{<ST>}.

\subsection*{Acknowledgements} The author thanks Professor P.~Salberger for providing the author with a copy of his preprint \cite{<S2>}. The author thanks his Doctoral Advisor Professor C.~L.~ Stewart for introducing him to this problem, many years of encouragement, and for his patient and thorough readings and corrections which improved the quality of this paper immeasurably. This work would not be possible if not for his efforts. Both anonymous referees provided very useful comments which significantly improved the quality of this paper. The author also thanks the University of Waterloo and the Government of Ontario for providing financial support while this work was being completed.

%%%%%%%%%%%%%%%%%%%%%%%%%%%%%%%%%%%
\section{Hilbert functions on weighted projective varieties}
\label{S2} 
%%%%%%%%%%%%%%%%%%%%%%%%%%%%%%%%%%%

In this section, we work out some basic notions of Hilbert functions and weighted homogeneous ideals needed for the rest of the paper. Salberger relied on the analogous results in the projective case for his results in \cite{<S1>}.\\ \\
Let $K$ be a fixed field of characteristic zero. We write $\boldsymbol \alpha = (\alpha_0, \cdots, \alpha_{r+1})$ to denote a sequence of non-negative integers, and for $\textbf{x} = (x_0, \cdots, x_{r+1})$ we write
\[\textbf{x}^{\boldsymbol \alpha} = x_0^{\alpha_0} \cdots x_{r+1}^{\alpha_{r+1}}.\]
Let $\mathbf{w} = (w_0, \cdots, w_{r+1})$ be a weight vector and let $u$ be a non-negative integer. For a monomial $\Bx^{\boldsymbol \alpha} = x_0^{\alpha_0} \cdots x_{r+1}^{\alpha_{r+1}}$, define the \emph{weighted degree} of $\Bx^{\boldsymbol \alpha}$ with respect to $\Bw$ to be
\[\boldsymbol \alpha \cdot \mathbf{w} = \alpha_0 w_0 + \cdots + \alpha_{r+1} w_{r+1}.\]
We say a polynomial $F \in K[x_0, \cdots, x_{r+1}]$ is \emph{weighted homogeneous} (with respect to $\mathbf{w}$) of weighted degree $u$ if for each monomial $\Bx^{\balpha}$ that appears in $F$ with a non-zero coefficient, the weighted degree of $\Bx^{\boldsymbol \alpha}$ is equal to $u$. This allows us to define the degree of a hypersurface $X$ in $\bP(\Bw)$, but not necessarily the degree of a subvariety of codimension greater than one. This will not be an issue since in our main application, we will embed such subvarieties explicitly into a lower dimensional weighted projective space, in which they will have codimension equal to one and so the definition for the hypersurface case applies. In other situations, we will rely on a pullback to a straight projective space where the notion of degree is well understood. \\

Define the set $K[x_0, \cdots, x_{r+1}]_{\textbf{w},u}$ to be the collection of weighted homogeneous polynomials with weight vector $\textbf{w}$ whose weighted degree is equal to $u$. We say that $I \subset K[x_0, \cdots, x_{r+1}]$ is a \emph{weighted homogeneous ideal} (with respect to \textbf{w}) if $I$ is generated by a set of weighted homogeneous polynomials with respect to the weight vector \textbf{w}. If $I \subset K[x_0, \cdots, x_{r+1}]_{\textbf{w}}$ is a weighted homogeneous ideal with weight vector $\textbf{w}$, then the set $I_u$ given by
$$\displaystyle I_u = I \cap K[x_0, \cdots, x_{r+1}]_{\textbf{w},u}$$
is a $K$-subspace of $K[x_0, \cdots, x_{r+1}]_{\textbf{w},u}$. Like in the projective case, we can define the Hilbert function of $I$ to be
$$ \displaystyle \mathcal{H}_I(u) = \dim_K(K[x_0, \cdots, x_{r+1}]_{\textbf{w},u} / I_u).$$ 
We can define a graded order $<$ on $K[x_0, \cdots, x_{r+1}]$ by the following: for $\boldsymbol \alpha = (\alpha_0, \cdots, \alpha_{r+1})$, $\boldsymbol \beta = (\beta_0, \cdots, \beta_{r+1}) \in \bZ_{\geq 0}^{r+2}$ we have $\boldsymbol \alpha > \boldsymbol \beta$ if $w_0 \alpha_0 + \cdots + w_{r} \alpha_{r} + w_{r+1} \alpha_{r+1} > w_0 \beta_0 + \cdots + w_{r} \beta_{r} + w_{r+1} \beta_{r+1}$. If there is a tie, i.e. $w_0 \alpha_0 + \cdots + w_{r+1} \alpha_{r+1} = w_0 \beta_0 + \cdots + w_{r+1} \beta_{r+1}$, then we take $\boldsymbol \alpha > \boldsymbol \beta$ if $\alpha_{r+1} - \beta_{r+1} > 0.$ If the weighted sums are equal and $\alpha_{r+1} = \beta_{r+1}$, then we compare $\alpha_{r}$ and $\beta_{r}$. This continues until we break the tie, so this ordering is a total order. Under this ordering, we can define the leading term of a given polynomial. 
\begin{definition} \label{2D1} Suppose 
\[F(x_0, \cdots, x_{r+1}) = \sum_{\textbf{w} \cdot \boldsymbol \beta = u} c_{\boldsymbol \beta} \textbf{x}^{\boldsymbol \beta} \in K[x_0, \cdots, x_{r+1}]\] 
is a weighted homogeneous polynomial with respect to the weight vector $\textbf{w}$ of weighted degree $u$. Suppose $\textbf{x}^{\boldsymbol \alpha}$ is a monomial which appears in $F$ with non-zero coefficient and which is maximal with respect to the total order $<$. Then, we say that $\textbf{x}^{\boldsymbol \alpha}$ is the \emph{leading monomial} of $F$. If we include the coefficient $c_{\boldsymbol \alpha}$ of $\textbf{x}^{\boldsymbol \alpha}$, then $c_{\balpha} \textbf{x}^{\boldsymbol \alpha}$ is the \emph{leading term} of $F$ which we write as $\operatorname{LT}(F)$. 
\end{definition}

Write $\langle \operatorname{LT}(I) \rangle$ to denote the ideal generated by the leading terms of polynomials in $I$. Our first result is the following:

\begin{proposition} \label{2P1} Let $I \subset K[x_0, \cdots, x_{r+1}]_{\Bw}$ be a weighted homogeneous ideal. Then $I$ has the same Hilbert function as $\langle \operatorname{LT}(I) \rangle$. \end{proposition} 

\begin{proof} The argument is identical to Proposition 9 in Chapter 9 of \cite{<CLO>}.
\end{proof}

\begin{remark} The choice of the ordering $<$ does not matter in Proposition \ref{2P1}. Indeed, we will choose slightly different orderings when convenient. \end{remark}
We have
\[ \mathcal{H}_I(u) = \mathcal{H}_{\text{LT}(I)}(u).\]
With this characterization, we can define for each $i \in \{0, 1, \cdots, r+1\}$
\begin{equation} \label{2E1} \displaystyle \sigma_{I,i}(u) = \sum_{\substack{\boldsymbol \beta \cdot \Bw = u \\ \textbf{x}^{\boldsymbol \beta} \notin \text{LT}(I)} } \beta_i. \end{equation}
From the definition of the Hilbert function, there are $\mathcal{H}_I(u)$ many monomials that are not the leading monomial of any polynomial in $I_u$. Thus, it follows immediately that
\[ w_0 \sigma_{I,0}(u) + \cdots + w_{r} \sigma_{I,r}(u) + w_{r+1}\sigma_{I,r+1}(u) = u \mathcal{H}_I(u).\]

Now by Theorem 3.4.4 in \cite{<Dol>}, the Hilbert series of a hypersurface generated by a form $F$ of weighted degree $d$ with respect to the weight vector $\textbf{w}$ is given by
\begin{equation} \label{Hilbert series} \frac{(1 - x^d)}{(1 - x^{w_0}) \cdots (1 - x^{w_{r+1}})}.\end{equation}
From here on, we shall assume that our weight vector $\Bw$ has the property that the $\gcd$ of any $r+1$ of the weights is equal to $1$. This distinction will be automatic in the relevant weight vectors in our paper; see Theorem \ref{T1}. Thus, by examining the poles of the function above we conclude that there is only one pole of order $r+1$, we see that the $u$-th coefficient is of the form
\begin{equation} \label{2E2} \displaystyle \mathcal{H}_I(u) = \frac{du^{r}}{r!w_0 \cdots w_{r+1}} + O_{\Bw,r}(d^{r+1} + d^2 u^{r-1}) = \frac{du^{r}}{r!w_0 \cdots w_{r+1}} + O_{\Bw,r}(d^{r+1} u^{r-1}) , \end{equation}
where the constant in front of the big-$O$ term depends only on $w_0, \cdots, w_{r+1}$ and $r$. \\ \\ 
The argument in the proof of our next result, Proposition \ref{2P2}, was inspired by a discussion on MathOverflow with Richard Stanley \cite{<Stan>}. In particular, the construction of the generating function used below was suggested by Stanley. 
\begin{proposition} \label{2P2} Let $K$ be a field of characteristic zero and $<$ be the graded monomial ordering as before. Suppose $F(x_0, \cdots, x_{r+1}) \in K[x_0, \cdots, x_{r+1}]$ has weighted degree $d$ with respect to $\Bw$ and leading monomial $\Bx^{\boldsymbol \alpha}$. Set $I = \langle F \rangle$. Define $\sigma_{I,m}(u)$ as in (\ref{2E1}). Then 
\[ \sigma_{I,m}(u) = a_{I,m} u\mathcal{H}_I(u)  + O_{\Bw, d, r}(u^{r} ),\]
where
\begin{equation} \label{2E4} a_{I,m} = \frac{d - w_m \alpha_m}{(r+1)w_m d}\end{equation}
for $m = 0, 1, \cdots, r+1$. 
\end{proposition}
\begin{proof} Suppose that $\textbf{x}^{\boldsymbol \beta}$ is a monomial of weighted degree $u$ with respect to the weight vector $\mathbf{w}$. Then $\textbf{x}^{\boldsymbol \beta} \in \langle \text{LT}(I) \rangle $ if and only if $\Bx^{\boldsymbol \alpha} | \textbf{x}^{\boldsymbol \beta}$. Hence, we need to count those monomials $\textbf{x}^{\boldsymbol \beta} = x_0^{\beta_0} \cdots x_{r+1}^{\beta_{r+1}}$ of weighted degree $u$ such that at least one of the exponents $\beta_i < \alpha_i$. Write $\sideset{}{^\ast}\sum$ to indicate a summation taken over those $\boldsymbol \beta = (\beta_0, \cdots, \beta_{r+1}) \in \bZ_{\geq 0}^{r+2}$ such that $w_0 \beta_0 + \cdots + w_{r+1} \beta_{r+1} = u$ and that $\beta_j < \alpha_j$ for some $0 \leq j \leq r+1$. Our goal, then, is to evaluate the sum
\[ \sigma_{I,m}(u) = \sideset{}{^\ast}\sum \beta_m\]
for each $0 \leq m \leq r+1$. To do this, let us define:
$$\displaystyle T_m^j(u) = \sum_{\substack{\boldsymbol \beta \cdot \Bw = u \\ \beta_j < \alpha_j}} \beta_m.$$
We want to emphasize that the evaluation of $T_m^j(u)$ will vary based on whether $j \ne m$ or $j = m$. In fact, the terms $T_m^m(u)$ will be negligible. We address the former situation. Define the function
$$\displaystyle G_{j,m}(x,y) = \frac{1 + y^{w_j} + \cdots + y^{w_j(\alpha_j - 1)}}{[\prod_{t \ne j,m} (1 - y^{w_t}) ](1-xy^{w_m})}$$
for $j \ne m$. We then take the derivative with respect to $x$ and evaluate at $x = 1$ to obtain
\begin{equation} \label{2E3} \frac{d}{dx} G_{j,m}(x,y) |_{x=1}  =  \frac{(1 + y^{w_j} + \cdots + y^{(\alpha_j - 1)w_j})y^{w_m}}{[\prod_{t \ne j,m} (1 - y^{w_t}) ](1-y^{w_m})^2}.\end{equation}
Note that $T_m^j(u)$ is equal to the coefficient of $y^u$ in the series expansion of (\ref{2E3}) around $0$. Since no $r+1$ of the weights have a common factor, it follows that for each root of unity $\zeta$, $\zeta$ is a root of at most $r+1$ factors in the denominator of (\ref{2E3}). Hence there is a single pole of order $r+2$ at $y = 1$.Since the highest order pole in (\ref{2E3}) is $r+2$, its Laurent series around $0$ is given by
\[c_{-r-2} y^{-r-2} + c_{-r-1} y^{-r-1} + \cdots\]
for complex coefficients $c_t \in \bC$. Using Cauchy's integral formula, we can calculate the coefficient $c_{-r-2}$:
\[ \frac{1}{2\pi i} \oint \frac{(1-z)^{r+1} (1 + z^{w_j} + \cdots + z^{w_j(\alpha_j - 1)})z^{w_j}}{(1-z^{w_m})^2\prod_{t \ne j,m} (1 - z^{w_t})} dz,\]
and get that
\[ c_{-r-2} = \frac{\alpha_j}{w_m^2 \prod_{t \ne j,m} w_t}.\]
Thus, $T_m^j(u)$ is asymptotically given by
$$\displaystyle \frac{\alpha_j}{w_m^2 \prod_{t \ne j,m} w_t} \frac{u^{r+1}}{(r+1)!}$$
for $j \ne m$, as $u \rightarrow \infty$. We now examine the contribution to $T_m^j(u)$ from other poles. From (\ref{2E3}), it follows that each pole is a root of unity. Recall that there are no other poles of order $r+2$. The contribution from a pole $\zeta$ of order $k$ less than $r+2$ is given by
\[ \frac{1}{2\pi i} \oint \frac{(\zeta-z)^{k-1} (1 + z^{w_j} + \cdots + z^{w_j(\alpha_j - 1)})z^{w_j}}{(1-z^{w_m})^2\prod_{t \ne j,m} (1 - z^{w_t})} dz.\]
The evaluation of this integral will depend on whether $\zeta$ is a $w_t$-th root of unity for $t \ne j$. To help us evaluate the integral, define
\[\ff_{\zeta,t}(z) = \begin{cases} \dfrac{1}{1 - z^{w_t}}, & \text{if } \zeta^{w_t} \ne 1, \\ \dfrac{\zeta - z}{1 - z^{w_t}}, & \text{if } \zeta^{w_t} = 1. \end{cases}\]
We now estimate $\ff_{\zeta, t}(\zeta)$ in both cases. Put $\zeta = e^{\frac{2\pi i l}{n}}$ with $\gcd(l,n) = 1$. Then
\begin{align*} 1 - \zeta^{w_t} & = 1 - \cos\left(\frac{2 \pi w_t l}{n} \right) - i \sin \left(\frac{2 \pi w_t l}{n} \right) \\
& = 2 \sin \left(\frac{\pi w_t l}{n}\right) \left(\sin \left(\frac{\pi w_t l}{n}\right) - i \cos \left(\frac{\pi w_t l}{n} \right) \right). \end{align*}
The term in the parentheses on the right has absolute value one, and we have
\[\left \lvert 2 \sin \left(\frac{\pi w_t l}{n} \right) \right \rvert \geq 2 \sin (\pi/n).\]
Moreover, $n \geq 2$, and on the interval $[0,\pi/2]$ $\sin(x)$ satisfies
\[\sin(x) \geq \frac{2x}{\pi},\]
whence
\[\left \lvert 2 \sin \left(\frac{\pi w_t l}{n} \right) \right \rvert \geq \frac{4}{n}.\]
Therefore, in this case, we have
\[\lvert \ff_{\zeta, t}(\zeta) \rvert \leq \frac{n}{4}.\]
In the second case, we put $\eta$ for a primitive $w_t$-th root of unity, and put $\zeta = \eta^l$ for some $1 \leq l \leq w_t - 1$. Then we make the observation that 
\[\prod_{\substack{1 \leq s \leq w_t \\ s \ne l}} (\zeta - \eta^s) = n \zeta^{n-1}.\]
Thus, in this case, we have $\lvert \ff_{\zeta, t}(\zeta) \rvert = n^{-1}$. \\ \\
Next, we deal with the numerator $\fg(z) = 1 + z^{w_j} + \cdots + z^{w_j(\alpha_j - 1)}$. We note that if $\fg(\zeta) \ne 0$, then we can simply bound from above by the triangle inequality to obtain the upper bound $\alpha_j$. Otherwise we make the observation that the contribution to the residue is equal to evaluating
\[\fg(z) (z - \zeta)^{-1}\]
at $z = \zeta$, which is equivalent to evaluating $\fg'(z) = w_j z^{w_j - 1} + \cdots + w_j(\alpha_j - 1) z^{w_j(\alpha_j - 1) - 1}$ at $z = \zeta$. The latter is readily seen to be bounded from above by $\displaystyle \frac{w_j \alpha_j (\alpha_j - 1)}{2}$. \\ \\
Combining these estimates, we see that order of magnitude of the residue does not exceed
\[\frac{w_j \alpha_j (\alpha_j - 1)}{2} n^{r + 2 - 2k}.\]
Therefore, the contribution to $T_m^j(u)$ from each pole of order $k$ is at most
\[\frac{w_j \alpha_j (\alpha_j - 1) n^{r + 2 - 2k}}{2} \frac{u^k}{k!}.\]
Note that $n$ is bounded above by the maximum of the $w_t$'s and bounded from below by the minimum of the $w_t$'s and $2$. Moreover, $\alpha_j$ is bounded from above by $d$. We have thus obtained an acceptable error term as stated in the proposition. \\ 

For the case $j = m$, we put
\[ G_{m,m}(x,y) = \frac{1 + xy^{w_m} + \cdots + (xy^{w_m})^{\alpha_m - 1}}{\prod_{t \ne m} (1 - y^{w_t})},\]
so that
\begin{equation} \label{LS2}\frac{d}{dx} G_{m,m}(x,y)|_{x=1} = \frac{y + 2y^2 + \cdots + (\alpha_m - 1)y^{\alpha_m - 1}}{\prod_{t \ne m} (1 - y^{w_t})}.\end{equation}
The pole at $y = 1$ is only of order $r+1$ as opposed to $r+2$. By examining the Laurent series of (\ref{LS2}) and evaluating the $-(r+1)$-th coefficient, we see that the contribution from the pole of order $(r+1)$ is equal to
\[ \frac{\alpha_m(\alpha_m - 1)}{\prod_{t \ne m} w_t} \frac{u^{r}}{r!}.\]
Observe that the coefficient is bounded from above by $d^2$. The lower order poles can be analyzed as before, so we omit this step. \\ \\
We now consider sums of the form 
$$\displaystyle \sideset{}{^\natural} \sum \beta_m$$
where the symbol $\sideset{}{^\natural} \sum$ indicates the sum is taken over those $\beta$ such that there exist at least two indices $i,j$ for which $\beta_i < \alpha_i$ and $\beta_j < \alpha_j$. Noting that $\alpha_j \leq d$ for $0 \leq j \leq r+1$ we see that the contribution from these sums is at most $C_3(\Bw,r) d^2 u^r$, where $C_3(\Bw,r)$ is a number which depends on $\Bw$ and $r$ only. The existence of such a $C_3(\Bw,r)$ follows from analyzing the order of poles as above and applying Cauchy's integral formula as above. Thus, by the inclusion exclusion principle, we see that for $0 \leq m \leq r+1$
\begin{align*} \displaystyle \sigma_{I,m}(u) & = \sum_{0 \leq j \leq r+1} T_m^j(u) + O_{\Bw,d,r} \left(u^r\right) \\
& = \frac{1}{w_m \prod_{t=0}^{r+1} w_t}\left(\frac{w_0 \alpha_0 u^{r+1}}{(r+1)!} + \cdots + \frac{w_{r+1} \alpha_{r+1} u^{r+1}}{(r+1)!} - \frac{w_m \alpha_m u^{r+1}}{(r+1)!} \right) + O_{\Bw,d,r}( u^{r}) \\
& = \frac{(d - w_m \alpha_m)u^{r+1}}{(r+1)! w_m \prod_{t=0}^{r+1} w_t } + O_{\Bw,d,r}(u^{r}).\end{align*}

Now, recall that $\displaystyle u\mathcal{H}_I(u) = \frac{du^{r+1}}{r!\prod_{t=0}^{r+1} w_t} + O_{\Bw,d,r} (u^{r}),$ and hence we have, for $0 \leq m \leq r+1$,
\[ \sigma_{I,m}(u) = \frac{d - w_m \alpha_m}{(r+1) w_m d} u \mathcal{H}_I(u) + O_{\Bw,d,r}\left(u^{r} \right).\]
This completes the proof of Proposition \ref{2P2}. \end{proof}

%%%%%%%%%%%%%%%%%%%%%%%%%%%%%%%%%%%%%%%%%%%%%
\section{The determinant method}
\label{S3}
%%%%%%%%%%%%%%%%%%%%%%%%%%%%%%%%%%%%%%%%%%%%%

In this section we lay out the necessary notation for our results and state our main technical theorem. From now on we will assume that the underlying field is $\bQ$, unless otherwise stated. For brevity we put $\bP(\Bw) = \bP_\bQ(w_0, \cdots, w_{r+1})$. We are not able to deal with general weighted projective spaces. Indeed, our arguments require at least two of the weights be equal to $1$. We shall assume that $w_0 = w_1 = 1$. This will be made apparent in the proof of Theorem \ref{T1}. \\ \\
Let $I$ be the weighted homogeneous ideal generated by a primitive weighted homogeneous form 
\[F(x_0, \cdots, x_{r+1}) \in \bZ[x_0, \cdots, x_{r+1}],\]
of weighted degree $d$, and let $X$ be the corresponding hypersurface defined by $F$. Let the \emph{height} of $F$, denoted by $\lVert F \rVert$, be the largest absolute value of the coefficients of $F$. Let $<$ be the monomial grading as in Section \ref{S2}, giving rise to the constants $a_{I,0}, \cdots, a_{I,r+1}$ as in (\ref{2E4}). Let $\BB = (B_0, \cdots, B_{r+1}) \in \bR^{r+2}$ be an $(r+2)$-tuple of real numbers of size at least $1$. Our goal is to count rational points $\Bx = (x_0, \cdots, x_{r+1})$ on the hypersurface $X$, defined over $\bP(\Bw)$, such that 
\[|x_i| \leq B_i, \text{ } 0 \leq i \leq r+1.\]
Let us write
\[ w = w_{2} \cdots w_{r+1}, \]
\begin{equation} \label{3E1} V = B_0 \cdots B_{r+1}, \end{equation}
and
\begin{equation} \label{3E2} W = \left(B_0^{a_{I,0}} \cdots B_{r+1}^{a_{I,r+1}}\right)^{\frac{r+1}{r} \left(\frac{w}{d}\right)^{1/r}}. \end{equation}
Further, we will only be concerned with those rational points $\Bx \in X$ with integral representation $(x_0, \cdots, x_{r+1})$ satisfying $\gcd(x_0, x_1) = 1$. Note that any such integral representative is necessarily primitive. Let us write $X(\bQ; B_0, \cdots, B_{r+1}) = X(\bQ; \BB)$ for the set of rational points on $X$ with an integral representative $(x_0, \cdots, x_{r+1})$ satisfying $|x_i| \leq B_i$ and $\gcd(x_0, x_1) = 1$. Sometimes we will wish to count a subset of $X(\bQ;\BB)$ satisfying a certain set of congruence conditions. For each prime $p$, let us write $X_p$ for the hypersurface defined by reducing $X$ modulo $p$, viewed as a variety over $\bF_p$. Let $\P = \{p_1, \cdots, p_t\}$ be a set of primes, and let $\mathfrak{P} = (P_1, \cdots, P_t)$, with $P_j \in X_{p_j}$. Then we write
\[X(\bQ;\BB; \mathfrak{P}) = \{\Bx \in X(\bQ;\BB) : \Bx \equiv P_j \pmod{p_j}, 1 \leq j \leq t\}.  \]
A hypersurface $X \subset \bP(\Bw)$ is \emph{geometrically integral} if it is reduced and irreducible over the algebraic closure of $\bQ$; see Hartshorne \cite{<Har>}, p. 82 and p. 93. 
\begin{theorem} \label{MT3} Let $\BB = (B_0, \cdots, B_{r+1}) \in \bR^{r+2}$ be a vector of positive numbers of size at least $1$ and let $\Bw = (1,1,w_2, \cdots, w_{r+1})$ be a vector of positive integers. Let $X$ be a hypersurface in $\bP(\Bw)$ which is irreducible over $\bQ$ and defined by a primitive weighted homogeneous form $F$ in $\bZ[x_0, \cdots, x_{r+1}]$ of weighted degree $d$ with respect to $\Bw$. Let $I = \langle F \rangle$ be the weighted homogeneous ideal generated by $F$. Let $\P$ be a finite set of primes and put 
\[\Q = \prod_{p \in \P} p. \]
For each prime $p$ in $\P$ let $P_p$ be a non-singular point in $X_p$ and put
\[\mathfrak{P} = \{P_p : p \in \P\}.\]
\begin{itemize}
\item[(a)] Let $\ep > 0$. If
\[WV^\ep \leq \Q \leq WV^{2\ep}\]
then there is a hypersurface $Y(\mathfrak{P})$ containing $X(\bQ; \BB, \mathfrak{P})$, not containing $X$ and defined by a primitive form $G \in \bZ[x_0, \cdots, x_{r+1}]$, whose weighted degree satisfies
\begin{equation} \label{degg} \deg G = O_{d,r,\Bw, \ep}(1),\end{equation}
and whose height satisfies 
\begin{equation} \label{heightg} \log \lVert G \rVert = O_{d,r, \Bw, \ep}\left(\log V \right).\end{equation}
\item[(b)]
If $X$ is geometrically integral, then there exists a hypersurface $Y(\mathfrak{P})$ containing $X(\bQ; \BB, \mathfrak{P})$, not containing $X$ and defined by a primitive form $G \in \bZ[x_0, \cdots, x_{r+1}]$, whose degree satisfies
\[ \deg G = O_{\Bw, d, r}\left((1 + \Q^{-1} W) \log V \Q \right).\]
\end{itemize}
\end{theorem}
The second part of Theorem \ref{MT3} is a generalization of Salberger's Theorem 2.2 in \cite{<S2>} to the case of weighted projective hypersurfaces, and the first part is a generalization of Salberger's Lemma 2.8 in \cite{<S2>}. Lemma 2.8 in \cite{<S2>} is itself an extension of Heath-Brown's Theorem 14 in \cite{<HB1>}. In fact, both theorems are recovered if we set $\Bw = (1,1, \cdots, 1)$. We note that, unlike earlier formulations when $\Q \geq WV^\ep$, the dependence of the logarithm of the height of $G$ on the degree $d$ and the dependence of the degree of $G$ on the degree $d$ of $F$ and the parameter $\ep$ is explicit with the remaining constant depending only on the dimension $r$ and the weight vector $\Bw$.  \\ \\
M.~ Walsh was able to obtain an improved version of Theorem 1.1 of \cite{<S2>} in \cite{<Wal>}. This corresponds to the case $\P = \emptyset$ in Theorem \ref{MT3}. His improvement was to show that one can obtain a saving of $\log (\lVert F \rVert +1) \lVert F \rVert^{-r^{-1} d^{-(r+1)/r}}$ on the estimate for the degree of the form $G$. \\ \\
Theorem \ref{MT3} is the main technical result of this paper. We will use it to carry out an inductive argument similar to Salberger's proof of Lemma 3.1 in his paper \cite{<S2>}. \\ \\
We will complete the proof of Theorem \ref{MT3} in the next three sections. 

%%%%%%%%%%%%%%%%%%%%%%%%%%%%%%%%%%%%%%%%%%%%%%%%
\section{Large divisors of the determinant}
\label{S4} 
%%%%%%%%%%%%%%%%%%%%%%%%%%%%%%%%%%%%%%%%%%%%%%%%

Our next theorem produces a prime power divisor of a determinant of the form $\det(M_j(\boldsymbol \xi_l))$, where $M_1, \cdots, M_s$ are monomials of the same weighted degree and where $\boldsymbol \xi_l \in \bZ^{r+2}$, $1 \leq l \leq s$ are all congruent to a point $P \in X_p$. The additional assumption that these tuples are congruent to some point $P \in X_p$ as opposed to the weaker assumption that they are merely congruent modulo $p$ gives the extra geometric information that allows us to produce a divisor which is larger. Indeed, if we assume only that $\boldsymbol \xi_l \equiv \boldsymbol \xi_j \pmod{p}$ for $1 \leq j, l \leq s$, then by taking differences of columns we can produce a factor of $p$ in each column, thereby allowing us to conclude that $p^{s-1} | \det(M_j(\boldsymbol \xi_l))$. However, our next theorem shows that for sufficiently large $s$, we can produce a larger power of $p$ which divides $\det (M_j(\boldsymbol \xi_l))$. We aim to establish the following:
\begin{theorem} \label{T1} Let $\Bw = (1,1,w_2, \cdots, w_{r+1})$ be a weight vector, $p$ be a prime, $X$ be a hypersurface of degree $d$ in $\bP(\Bw)$, and $P$ be an $\bF_p$ point of multiplicity $m_P$ on $X_p$. Suppose there are $s$ distinct  primitive $(r+2)$-tuples of integers on $X$
\[\boldsymbol\xi_1, \cdots, \boldsymbol\xi_s\] 
with reduction $P$, such that $\gcd(\xi_{0,l}, \xi_{1,l})= 1$ for $1 \leq l \leq s$. If $M_1, \cdots, M_s$ are monomials in $(x_0, \cdots, x_{r+1})$ of the same weighted degree, then there exists a positive number $\kappa(d,r)$, depending on $d$ and $r$, such that the determinant of the $s \times s$ matrix $(M_j(\boldsymbol \xi_l))$ is divisible by $p^N$, where
\[N > \left(\frac{r!}{m_P}\right)^{\frac{1}{r}} \cdot \frac{r}{r+1}\cdot s^{1 + \frac{1}{r}} - \kappa(d,r)s.\]
If $P$ is non-singular, so $m_P = 1$, then there exists a positive number $\kappa'(r)$, depending only on $r$, such that 
\[N > (r!)^{1/r} \frac{r}{r+1} s^{1 + \frac{1}{r}} - \kappa'( r)s.\]
\end{theorem}
We will prove Theorem \ref{T1} by means of the next two propositions; corresponding to Lemmas 2.3 and 2.4 respectively in \cite{<S1>}. We note here that for the proof of Theorem \ref{T1} we require that two of the weights be $1$. This is the only part of the paper where we need to make such an assumption. \\

We remark that this restriction can be removed if we a priori pick monomials whose weighted degrees are a multiple of the least common multiple of all of the weights, and indeed this opens up the possibility to extend the determinant method to all weighted projective spaces. However the extra technical details take us too far afield in the present paper. We would like to return to this issue in the future. 

\begin{proposition} \label{4P1} Let $\Bw = (1,1,\cdots, w_{r+1})$ be a weight vector, $X$ a hypersurface of weighted degree $d$ in $\bP(\Bw)$, $p$ a prime and $P$ an $\bF_p$-point of multiplicity $m_P$ on $X_p$. Write $A$ for the local ring of regular functions at $P$ and $\fm$ for the maximal ideal of $A$. For each positive integer $t$ put $g_{X,P}(t) = \dim_{A/\fm} \fm^t / \fm^{t+1}$. Then, we have 
\[g_{X,P}(t) = \frac{m_P t^{r-1}}{(r-1)!} + O_{ d,r}( t^{r-2}).\]
If $m_P = 1$, then we obtain the more refined assertion that
\[g_{X,P}(t) = \frac{t^{r-1}}{(r-1)!} + O_{r}(t^{r-2}).\]
\end{proposition}

\begin{proof} Write $\B = \displaystyle \bigoplus_{t \geq 0} (\fm^t / \fm^{t+1})$. By definition, the projectivized tangent cone at $P$ is defined to be the $\operatorname{Proj}(\B)$, see Exercise III-29 in \cite{<EH>}. Since $A/\fm \cong \bF_p$ is a field, it follows that $g_{X,P}(t)$ is precisely the Hilbert function of the projectivized tangent cone at $P$, say $W_P$. Note that $W_P$ is a subvariety of the Zariski tangent space of $X$ at $P$, which is isomorphic to $\bP_{\bF_p}^r$. Hence, we can consider the homogeneous ideal of $W_P$, which is generated by $C_4(d, r)$ many forms; see III.3 of \cite{<Mum>}. Note that this bound depends only on $d$ and $r$. Following Lemma 1 of \cite{Bro}, we may choose a Groebner basis of forms of degree $C_5( d, r)$ for the homogeneous ideal of $W_P$. By Proposition \ref{2P1}, the Hilbert function does not change if we replace this ideal with the ideal generated by its leading terms. Hence, there are only finitely many candidates for Hilbert functions of $W_P$ for points $P$ of multiplicity $m_P = O_{\Bw, d}(1)$. More precisely, the number of candidates is bounded by the number of monomials in $r-1$ variables of degree at most $C_5(d, r)$. Thus, there are at most $C_6(d, r)$ such functions.  \\ \\
Let us now fix a particular 
\[g_{X,P}(t) = \frac{m_P t^{r-1}}{(r-1)!} + O_{P,r}(t^{r-2}).\]
To obtain the estimate for the coefficient in front of the big-$O$ term, one notes that there exists a polynomial $Q(x)$ with integer coefficients with $Q(1) \ne 0$ such that the Hilbert series of the projectivized tangent cone is given by
\[\frac{Q(x)}{(1 - x)^{r}}, \]
see Chapter 9 of \cite{<CLO>}. From here we see from Proposition \ref{2P2} that the error term is at most an absolute constant times $m_P^{r-1}$. Since $m_P = O_{\Bw, d}(1)$, the claim follows. \\ \\
If $m_P = 1$, then it is known (see III.3 in \cite{<Mum>}) that the ideal of the tangent cone at $P$ is generated by a single polynomial of degree $1$. Hence, we can replace $C_4(d, r), C_5(d, r)$, and $C_6(d, r)$ with numbers that depend at most on $r$. \end{proof}
We shall denote by $\bZ_p$ the ring of $p$-adic integers. Let $R$ be a commutative noetherian local ring containing $\bZ_p$ as a subring, $\R = R/pR$, and $\mathfrak{m}$ be the maximal ideal of $\R$. We then have the following proposition:
\begin{proposition} \label{4P2} Let $(n_l(\R))_{l = 1}^\infty$ be the non-decreasing sequence of integers $t \geq 0$, where $t$ occurs exactly $\dim_{\R/\mathfrak{m}} \mathfrak{m}^t/\mathfrak{m}^{t+1}$ times. Let $r_1, \cdots, r_s$ be elements of $R$ and $\phi_1, \cdots, \phi_s$ be ring homomorphisms from $R$ to $\bZ_p$. Then, the determinant of the $s \times s$ matrix $(\phi_i(r_j))$ is divisible by $p^{A(s)}$ for $A(s) = n_1(\R) + \cdots + n_s(\R)$. 
\end{proposition}
\begin{proof} This is the same as the proof of Lemma 2.4 in \cite{<S1>}.\end{proof}

\begin{proof} (Theorem \ref{T1}) Let $R$ be the local ring of $X$ over $\bZ_p$ at the point $P$ with respect to the weight vector $\Bw = (1, 1, w_{2}, \cdots, w_{r+1})$ and $\R = R/pR$. Since $\gcd(x_0, x_{1}) = 1$, there exists some index $j = 0,1$ such that $p \nmid x_j$. Without loss of generality, suppose that $p \nmid x_0$. Then we can replace $M_j(x_0, \cdots, x_{r+1})$ with 
\[M_j\left(1, \frac{x_1}{x_0},\frac{x_{2}}{x_0^{w_2}}, \cdots, \frac{x_{r+1}}{x_0^{w_{r+1}}}\right)\]
without changing the $p$-adic valuation of $\det(M_j(\boldsymbol \xi_l))$. These rational functions are elements of $R$. We consider the evaluation maps at the points $\boldsymbol \xi_1, \cdots, \boldsymbol\xi_s$, which are homomorphisms  from $R$ to $\bZ_p$. Since $\bZ_p \subset R$, the conditions for the ring appearing in Proposition \ref{4P2} is satisfied. Thus it follows that
\[ p^{A(s)} | \Delta.\]
It remains to estimate $A(s)$. Let $g = g_{X,P}$ be as in proposition \ref{4P1} and set $G(t) = g(0) + g(1) + \cdots + g(t).$ Since $g(t) = m_P t^{r-1}/(r-1)! + O_{d,r}(t^{r-2}),$ it follows that
\[ G(t) = \frac{m_P t^{r}}{r!} + O_{d,r}( t^{r-1}).\]
By the definition of $g$ and $(n_l(\R))$, it follows that
\[ A(G(t)) = g(1) + \cdots + tg(t) = \frac{m_P t^{r+1}}{(r+1)(r-1)!} + O_{d,r}( t^{r}),\]
and explicitly we have
\begin{align*} \left(\frac{r!}{m_P}\right)^{\frac{1}{r}} G(t)^{1 + \frac{1}{r}} & = \left(\frac{r!}{m_P}\right)^{\frac{1}{r}} \left(\frac{m_P t^{r}}{r!} + O_{d,r}( t^{r-1})\right)^{1 + \frac{1}{r}} \\
& =  \frac{m_P t^{r+1}}{r!} + O_{d,r}( t^{r}). \end{align*}
Multiplying by $r/(r+1)$ gives
\[ A(G(t)) = \left(\frac{r!}{m_P}\right)^{\frac{1}{r}} \left(\frac{r}{r+1}\right) G(t)^{1 + \frac{1}{r}} + O_{d,r}( G(t)),\]
since $t^{r} = O_{\Bw, d, r}( G(t)).$ The fact that
\[ A(s) = \left(\frac{r!}{m_P}\right)^{\frac{1}{r}} \left(\frac{r}{r+1}\right) s^{1 + \frac{1}{r}} + O_{d,r}( s)\]
follows from the observation that if $t$ is the unique integer such that $G(t-1) < s \leq G(t)$, then 
\[ 0 \leq A(G(t)) - A(s) \leq tg(t) \leq \frac{m_P t^{r}}{(r-1)!} + O_{d,r}( t^{r-1}) \leq rs + O_{d,r}( s^{1 - \frac{1}{r}}),\]
and
\[ 0 \leq G(t)^{1 + \frac{1}{r}} - s^{1 + \frac{1}{r}} \leq G(t)^{1 + \frac{1}{r}} - G(t-1)^{1 + \frac{1}{r}} = O_{d,r}( t^{r}) = O_{d,r}( s).\] 
If $m_P = 1$, then by Proposition \ref{4P1} the constants in front of the error terms may be replaced with a number which depends on $r$ only. 
\end{proof}

We now proceed to give estimates for products of various `bad' primes with respect to a geometrically integral hypersurface $X \subset \bP(\Bw)$. 
\begin{definition} \label{D3} Let $X$ be a geometrically integral hypersurface in $\bP(\Bw)$ of degree $d$. We write $\pi_X$ for the product of all primes $p$ for which $X_p$ is not geometrically integral. 
\end{definition}
Let us denote by $R_{r+1}(d)$ the number of distinct monomials in $x_0, \cdots, x_{r+1}$ of weighted degree $d$ with respect to the weight vector $\Bw = (1, 1, w_{2}, \cdots, w_{r+1})$. \\ \\
The next lemma allows us to capture whether a given polynomial is irreducible over $\ol{\bQ}$ or not by considering a finite set of universal polynomials. This was first proved by Salberger in \cite{<S1>}, and Lemma \ref{5Lem2} below is essentially the same as Lemma 1.8 in \cite{<S2>}, except over weighted projective space. \\ \\
Denote by $\S_d$ the set of vectors $\boldsymbol \beta \in \bZ_{\geq 0}^{r+2}$ such that $\boldsymbol \beta \cdot \Bw = d$. Note that
\[\# \S_d = R_{r+1}(d).\]
Let the elements in $\S_d$ be enumerated by $\bbeta_1, \cdots, \bbeta_{R_{r+1}(d)}.$ 
\begin{lemma} \label{5Lem2} Let $d$ be a positive integer. Then there exists a finite set of universal forms 
\[ \Phi_1(a_{1}, \cdots, a_{{R_{r+1}(d)}}), \cdots, \Phi_t(a_{1}, \cdots, a_{{R_{r+1}(d)}}), \]
with the following property. Whenever the variables $a_{j}$ take values in a field $K$, the form 
\[ F(x_0, \cdots, x_{r+1}) = \sum_{j=1}^{R_{r+1}(d)} a_{j} \Bx^{\bbeta_j} \]
is absolutely irreducible over $K$ if and only if $\Phi_i(a_{1}, \cdots, a_{{R_{r+1}(d)}}) \ne 0$ in $K$ for some $i \in \{1, \cdots, t\}$.
\end{lemma}

\begin{proof} First, we remark that weighted projective space can be realized as an abstract projective scheme by considering a grading corresponding to its weight vector. See Miles Reid's course notes \cite{<Re>}. Thus, let $\bH_k$ denote the Hilbert scheme of degree $k$ hypersurfaces in $\bP(\Bw)$. Since these hypersurfaces are defined by polynomials of degree $k$, there is a natural morphism between $\bH_k \times \bH_{d-k}$ and $\bH_d$. Let $v_k$ denote this morphism. Then, 
\[F(x_0, \cdots, x_{r+1}) = \sum_{j=1}^{R_{r+1}(d)} a_{j} \Bx^{\bbeta_j}\]
has a factor over $K$ of degree $k$ if and only if the corresponding $K$-point on $\bH_d$ lies in $v_k(\bH_k \times \bH_{d-k})$. Also, since $\bH_k \times \bH_{d-k}$ is a projective scheme, $v_k(\bH_k \times \bH_{d-k})$ must be a closed subset of $\bH_d$ by the main theorem in elimination theory in Chapter 3, Section 1 of \cite{<CLO>}. The union of $v_k(\bH_k \times \bH_{d-k})$ over $k = 1, \cdots, d-1$ must be a closed subset of $\bH_d$ defined by a finite set of forms 
\[\Phi_1(a_{1}, \cdots, a_{{R_{r+1}(d)}}), \cdots, \Phi_t(a_{1}, \cdots, a_{{R_{r+1}(d)}})\]
over $\bZ$ such that $F$ is absolutely irreducible over $K$ if and only if $\Phi_i(a_{1}, \cdots, a_{{R_{r+1}(d)}}) = 0$ for all $1 \leq i \leq t$ in $K$. This completes the proof. 
\end{proof}

%\begin{remark} By \cite{<Mac>}, the Hilbert scheme of a projective scheme depends only on the grading used to order monomials and so it is independent of the underlying field. Since we have chosen a grading using $\bN$, the resulting universal polynomials in Lemma \ref{5Lem2} will have integral coefficients. See Theorem 1.1 in \cite{<Har1>}. 
%\end{remark}

The next lemma gives an upper bound for $\pi_X$ in the case when $X(\bQ; \BB)$ is not contained in another hypersurface of the same degree as $X$. 

\begin{lemma} \label{5Lem3} Let $X \subset \bP(\Bw)$ be a geometrically integral hypersurface of degree $d$ and $\BB = (B_0, \cdots, B_{r+1}) \in \bR_{\geq 1}^{r+2}$. Then one of the following statements hold: 
\begin{enumerate}
\item[(a)] \label{5L3A} $X(\bQ; \BB)$ lies in a hypersurface $Y \ne X$ of degree $d$,
\item[(b)] \label{5L3C} $\log \pi_X = O_{\Bw, d, r}\left(1 + \log V\right)$. 
\end{enumerate}
\end{lemma}

\begin{proof} Let $F(x_0, \cdots, x_{r+1}) = \sum_{j=1}^{R_{r+1}(d)} a_{j} \Bx^{\bbeta_j}$ be a primitive integral form defining $X$ and 
\[\Phi_1(a_{1}, \cdots, a_{{R_{r+1}(d)}}), \cdots, \Phi_t(a_{1}, \cdots, a_{{R_{r+1}(d)}})\]
be the values of the universal forms in Lemma \ref{5Lem2} of the coefficients $a_{j}$ of $F$. Then $\Phi_i(a_{1}, \cdots, a_{{R_{r+1}(d)}}) \ne 0$ for some $i \in \{1, \cdots, t\}$, as $X$ is geometrically integral. By applying Lemma \ref{5Lem2} to $F_p$, which is $F$ reduced modulo $p$, and setting $K = \ol{\bF_p}$ for the prime factors $p$ of $\Phi_i(a_{1}, \cdots, a_{{R_{r+1}(d)}})$, we see that $\pi_X$ is a factor of $\Phi_i(a_{1}, \cdots, a_{{R_{r+1}(d)}})$. Note that the degree $D$ of $\Phi_i$ and the height $\lVert \Phi_i \rVert$ are bounded in terms of $d$ and $r$. Write $S = \# X(\bQ; \BB)$ and $s = R_{r+1}(d)$. Form the $S \times s$ matrix $\M$, where the rows correspond to the points $\textbf{x}_1, \cdots, \textbf{x}_S \in  X(\bQ; \BB)$ and the columns correspond to the monomials of weighted degree $d$. Then the vector $\textbf{f} \in \bZ^s$ corresponding to the coefficients of $F$ satisfies $\M \textbf{f} = \textbf{0}$, whence the rank of $\M$ is at most $s-1$. Let $s' \leq s-1$ denote the rank of $\M$. Then, for any $(s'+1) \times (s'+1)$ minor $\M'$ of $\M$, we have $\det \M' = 0$, while there exists some $s' \times s'$ minor $\M''$ of $\M$ such that $\det \M'' \ne 0$. Without loss of generality, assume that $\M''$ consisting of the first $s'$ columns and $s'$ rows of $\M$ is such that $\det \M'' \ne 0$. Then, by taking the $(s'+1) \times (s'+1)$ minor $\M'$ consisting of the first $s'+1$ columns and $s'+1$ rows of $\M$, we have that
\begin{equation} \label{mprime} \det \M' = 0.\end{equation}
Expanding $\det \M'$ along the right most column of $\M'$, we see that (\ref{mprime}) implies that there exists an integral vector $\textbf{g} \in \bZ^s$, whose entries are at most $V^{ds}$, such that $\M \textbf{g} = \textbf{0}$. Let $G$ be the corresponding weighted form. Note that $G$ is not the zero form and has degree $d$. Further, $G$ vanishes on $X(\bQ; \BB)$. Hence, if (a) does not hold, $G$ must be a constant multiple of $F$. Thus, it follows that 
\begin{equation} \label{4EE2} \lVert F \rVert \ll (R_{r+1}(d))! V^{dR_{r+1}(d)}\end{equation}
where the implied constant is absolute. Therefore, there exists $C_{7}(\Bw, d, r)$ such that 
\[\lvert \Phi_i(a_{1}, \cdots, a_{{R_{r+1}(d)}}) \rvert = O_{\Bw, d, r}(V^{C_7(\Bw, d, r)}).\] 
Since $\pi_X$ divides $\Phi_i(a_1, \cdots, a_{R_{r+1}(d)}),$ we have 
\[ \log \pi_X = O_{\Bw, d, r}( 1 + \log V)\]
if (a) does not hold, as desired.
\end{proof}

%%%%%%%%%%%%%%%%%%%%%%%%%%%%%%%%%%%%%%%%%%%%%%%%
\section{Proof of Theorem \ref{MT3}: Preliminaries}
\label{S5} 
%%%%%%%%%%%%%%%%%%%%%%%%%%%%%%%%%%%%%%%%%%%%%%%%

In the next two sections we complete the proof of Theorem \ref{MT3}. We have chosen to give arguments similar to those given by Salberger to prove his Lemma 1.4 in \cite{<S2>}, which is stated as Lemma \ref{6Lem1}  below. The argument in the proof of Lemma \ref{6Lem1} is essentially the same as the proof of Lemma 1.4 in \cite{<S2>}; Walsh also proved a similar result in \cite{<Wal>}. \\ \\
For a given point $P$ on $X_p$ let $m_P$ denote the multiplicity of $P$. Next, let us write $n_p = \sum_P m_P$, where the sum is over all points $P \in X_p$. 

\begin{lemma} \label{6Lem1} Let $X$ be a geometrically integral hypersurface in $\bP(\Bw)$ of degree $d$ defined by a primitive form $F$, and let $p$ be a prime for which $X_p$ is geometrically integral.  Suppose there exist $s$ primitive $(r+2)$-tuples of integers
$$\boldsymbol\xi_1, \cdots, \boldsymbol\xi_s$$
representing elements of $X(\bQ; \BB)$. Let $M_1, \cdots, M_s$ be monomials in $(x_0, \cdots, x_{r+1})$ with integer coefficients and the same weighted degree. Then, there is a positive number $\kappa(d,r)$ which depends on $d$ and $r$, such that the determinant of the $s \times s$ matrix formed by the entries $M_j(\boldsymbol\xi_l)$ is divisible by $p^{N}$ with
\[ N > (r!)^{1/r} \frac{r}{r+1} \frac{s^{1 + 1/r}}{n_p^{1/r}} - \kappa(d,r) s. \]
\end{lemma} 
\begin{remark} The number $\kappa(d,r)$ is the same as in Theorem \ref{T1}.
\end{remark}
\begin{proof} Let $P$ be an $\bF_p$-point on $X_p$. Write $I_P \subset \{1, \cdots, s\}$ for the set of indices $l$ such that $\boldsymbol \xi_l + p \bZ^{r+2}$ represents $P$, and write $s_P = \# I_P$. Then, by Theorem \ref{T1}, there exists a non-negative integer 
\begin{equation} \label{6E1} N_P > \left(\frac{r!}{m_P}\right)^{1/r} \frac{r}{r+1} s_P^{1 + 1/r} - \kappa(d,r) s_P, \end{equation}
such that $p^{N_P} | \det(\M_P)$, where $\M_P$ is a $s_P \times s_P$ submatrix of $\M$ with second indices $l \in I_P$. By Laplace expansion, we can express $\Delta$ as follows:
\[ \Delta = \sum \operatorname{sgn}(\M_P, \M_P')\det(\M_P) \det(\M_P'),\]
where the sum is over all $s_P \times s_P$ minors $\M_P$ along the indices in $I_P$ and $\M_P'$ is the complementary minor of $\M_P$. We can iterate this process with each $\M_P'$, which consists of rows with indices in the set $\{1, \cdots, s\} \setminus I_P$. Each iteration yields a divisor of $\Delta$ which is independent of $p^{N_P}$. Hence, we get that $p^N | \Delta$, where 
\[N = \sum_P N_P > (r!)^{1/r} \frac{r}{r+1} \sum_P \frac{s_P^{1 + 1/r}}{m_P^{1/r}} - \kappa(d,r) s.\]
By H\"older's inequality, we get that
\[ s = \sum_P s_P \leq \left(\sum_P m_P \right)^{1/(r+1)} \left( \sum_P \frac{s_P^{1 + 1/r}}{ m_P^{1/r}} \right)^{r/(r+1)}.\]
Re-arranging, we obtain
\[ \sum_P \frac{s_P^{1 + 1/r}}{m_P^{1/r}} \geq \frac{s^{1 + 1/r}}{n_p^{1/r}}.\]
Thus, we have that
\[ N \geq (r!)^{1/r}\frac{r}{r+1}\frac{s^{1 + 1/r}}{n_p^{1/r}} - \kappa(d,r) s,\]
as desired. \end{proof}
We now draw on some results of Lang and Weil in \cite{<LW>} on the number of points of algebraic varieties over finite fields. Let us define $X_{p, \sing}$ to be the singular locus of $X_p$. Let $X_{p,j}$ be the zero locus of the partial derivative $\displaystyle \frac{\partial F}{\partial x_j}$ over $\bF_p$. Then $X_{p, \sing} \subset X_p \cap X_{p,j}$ for each $j = 0, \cdots, r+1$. In particular, $X_{p, \sing}$ has co-dimension at least one in $X_p$ since the partial derivatives of $F$ do not all vanish identically. By example 4 on page 130 of \cite{BR}, both $X_p$ and $X_{p,j}$ arise as quotients under the same action of hypersurfaces of degree $d$ and $d-1$ respectively over $\bP^{r+1}(\bF_p)$, thus the usual B\'{e}zout's theorem gives an upper bound for the number of components in $X_{p, \sing}$ as well as its degree. Therefore, the sum of the degrees of the irreducible components of $X_{p, \sing}$ is bounded in terms of $d$ and $r$. Hence, by Lemma 1 in \cite{<LW>}, we have $\# X_{p, \sing}(\bF_p) = O_{\Bw, d,r}(p^{r-1})$. Since the multiplicity of a point on $X_p$ is bounded in terms of $d$, it follows that
%Next we count the set of points on $X_{p,\sing}$ where $x_0 = x_1 = 0$. Even over $\bP(\Bw)(\bF_p)$, there are only $O_{\Bw, d,r}(p^{r-1})$ such points and so there are only at most this many points of this type on $X_{p,\sing}$. 

\[ \sum_P (m_P - 1)= O_{\Bw, d,r}(p^{r-1}).\]
Next we examine the sizes of equivalence classes for each point $\Bx$ in $\bP_{\bF_p}(\Bw)$. Suppose $\Bx = (x_0, \cdots, x_{r+1})$. If either $x_0$ or $x_1$ is non-zero in $\bF_p$, then we see that for $\lambda \not \equiv \gamma \pmod{p}$ the points
\begin{equation} \label{weight 1} (\lambda x_0, \lambda x_1, \lambda^{w_2} x_2, \cdots, \lambda^{w_{r+1}} x_{r+1})\end{equation}
and
\begin{equation} \label{weight 2} (\gamma x_0, \gamma x_1, \gamma^{w_2} x_2, \cdots, \gamma^{w_{r+1}} x_{r+1})\end{equation}
are distinct in $\bA_{\bF_p}^{r+2}$. Therefore, the equivalence class of $\Bx$ in $\bP_{\bF_p}(\Bw)$ has size $p-1$. In general, if (\ref{weight 1}) and (\ref{weight 2}) are in the same equivalence class in $\bP_{\bF_p}(\Bw)$, then the congruence
\[\lambda^{w_j} \equiv \gamma^{w_j} \pmod{p}\]
has to hold for each $j$ such that $x_j$ is non-zero, by the definition of weighted projective space. Thus, if the non-zero coordinates of $\Bx$ have co-prime weights then (\ref{weight 1}) and (\ref{weight 2}) are distinct whenever $\lambda \not \equiv \gamma \pmod{p}$. If the weights of the non-zero coordinates of $\Bx$ are not co-prime, then let $g$ be their $\gcd$. Without loss of generality, we may suppose that the non-zero coordinates are $x_2, \cdots, x_k$ for some $k \leq r+1$. For each non-zero $\lambda$ in $\bF_p$, choose a $g$-th root $\lambda^{1/g}$ in $\ol{\bF_p}$. Then we see that
\[(0,0, \lambda^{w_2/g} x_2, \cdots, \lambda^{w_k/g} x_k, \cdots)\]
and
\[(0,0, \gamma^{w_2/g} x_2, \cdots, \gamma^{w_k/g} x_k, \cdots)\]
are distinct in $\bA_{\bF_p}^{r+2}$ whenever $\lambda \not \equiv \gamma \pmod{p}$. Therefore, each equivalence class in $\bP_{\bF_p}(\Bw)$ contains exactly $p-1$ elements. Put $X'(\bF_p)$ to be the variety in $\bA_{\bF_p}^{r+2}$ defined by the same polynomial as $X$. It then follows that
\[\# X(\bF_p) = (p-1) \# X'(\bF_p),\]
so Theorem 1 of \cite{<LW>} implies that $\# X(\bF_p) = p^r + O(d^2 p^{r-1/2}) + O_{\Bw, d, r}\left(p^{r-1}\right)$, hence
\[ n_p = p^r + O(d^2 p^{r-1/2}) + O_{\Bw, d, r}\left(p^{r-1}\right).\]

%Now note that Lang and Weil's result in \cite{<LW>} deal with projective varieties, not varieties in weighted projective space. However, since the weighted projective spaces we are considering have two weights $w_0 = w_1 = 1$, aside from the $O(p^{r-1})$ points on $X_p$ with $x_0 = x_1 = 0$, the remaining points are in one to one correspondence with points on an affine variety where one may apply the Lang-Weil estimate. Thus the set of points whose equivalence classes have size less than $p-1$ are negligible. \\

More specifically, the implied constant is at most $d^2$ by the argument in \cite{<LW>}. Note that for all real numbers $\alpha \in \bR_{\geq 0}$, we have $\alpha - 1 = (\alpha^{1/r} - 1)(\alpha^{(r-1)/r} + \cdots + 1),$ and so $|\alpha^{1/r} - 1| \leq |\alpha - 1|$. Thus, $n_p^{1/r}  = p + O_{\Bw,d,r}(p^{1/2})$. We summarize this as a lemma:
\begin{lemma} \label{6Lem2} If $X_p$ is geometrically integral, then $n_p^{1/r} = p + O_{\Bw,d,r}(p^{1/2}).$
\end{lemma}
We are now ready to complete the proof of Theorem \ref{MT3}. 

%%%%%%%%%%%%%%%%%%%%%%%%%%%%%%%%%%%%%%%%%%%%%%%%%%
\section{Proof of Theorem \ref{MT3}: Completion}
\label{S6} 
%%%%%%%%%%%%%%%%%%%%%%%%%%%%%%%%%%%%%%%%%%%%%%%%%%

Let $S = \# X(\bQ; \BB; \mathfrak{P})$, and let 
\[\boldsymbol \xi_1, \cdots, \boldsymbol \xi_S \in X(\bQ;\BB, \mathfrak{P})\]
be primitive integral $(r+2)$-tuples representing elements of $X(\bQ;\BB; \mathfrak{P})$. Let $I$ be the weighted homogeneous ideal generated by $F$ in $\bZ[x_0, \cdots, x_{r+1}]$. For a positive integer $s$, let $u$ be the positive integer such that $\mathcal{H}_I(u-1) < s \leq \mathcal{H}_I(u)$. By (\ref{2E2}), it follows that
\[s = \frac{du^r}{r! w} + O_{\Bw,r}\left(d^{r+2} + d^2 u^{r-1}\right) = \frac{du^r}{r! w} \left(1 + O_{\Bw,r}\left( d^{r+1} u^{-r} +  d u^{-1} \right) \right), \]
hence
\[\left(\frac{w \cdot r!}{d}\right)^{1/r} s^{1/r} = u \left(1 + O_{\Bw,r}( d^{r+1}u^{-r}  + d u^{-1})\right)^{1/r}. \]
Rearranging, we obtain
\begin{equation} \label{6E0} u = \left(\frac{w \cdot r!}{d}\right)^{\frac{1}{r}} s^{\frac{1}{r}} + O_{\Bw,r}\left(d \right).
\end{equation}
Observe that 
\[\mathcal{H}_I(u) - \mathcal{H}_I(u-1) = \frac{d(u^r - (u-1)^r)}{r! w} + O_{\Bw,d,r}(u^{r-1}) = O_{\Bw,d,r}(u^{r-1}),\]
hence by our choice of $u$ with respect to $s$, we have
\[\mathcal{H}_I(u) - s = O_{\Bw,d,r}\left(u^{r-1}\right).  \]
Therefore, 
\begin{equation} \label{hilbert} \end{equation} 
\begin{align*} u \mathcal{H}_I(u) & = \left(\left(\frac{w \cdot r!}{d}\right)^{1/r} s^{1/r} + O_{\Bw,r}(d) \right)\left(s + O_{\Bw,d,r}(u^{r-1}) \right) \\
& =  \left(\frac{w \cdot r!}{d}\right)^{1/r} s^{1 + \frac{1}{r}} + O_{\Bw,d,r}\left(s\right). \end{align*}
Let $M_1, M_2, \cdots, M_s$ be distinct monomials of weighted degree $u$ which are not leading monomials of any element in $I = \langle F \rangle$. These monomials are linearly independent over $\bQ$, and any $\bQ$-linear combination of them is not a multiple of $F$. \\ \\
Set
\[\mathfrak{M} = (M_j(\boldsymbol \xi_l))_{\substack{1 \leq j \leq s \\ 1 \leq l \leq S}}.\]
If $S < s$, then $\mathfrak{M}$ has rank at most $s-1$. Hence, $\mathfrak{M}$ has a non-trivial kernel, so there exists a vector $\mathbf{g} \in \bZ^s$ such that 
\[\mathfrak{M} \mathbf{g} = \mathbf{0}. \]
Such a $\mathbf{g}$ gives rise to a form $G$ such that for all $\Bx \in X(\bQ; \BB; \mathfrak{P})$, we have $G(\Bx) = 0$. Thus $G$ defines a hypersurface $Y$ such that $X(\bQ; \BB; \mathfrak{P}) \subset Y$ and $\deg Y = u$. \\ \\
We now assume that $S \geq s$. If we can prove that for any $s \times s$ minor $\M$ of $\mathfrak{M}$ has determinant equal to $0$, then $\mathfrak{M}$ has rank at most $s - 1$. This is the goal we devote the rest of this section to. We choose, as we may, $\M$ to be the $s \times s$ minor of $\mathfrak{M}$ composed of the first $s$ rows, and consider
\[\Delta = \det \M.\]
We estimate $\Delta$ from above as follows:
\[ |\Delta| \leq s! B_0^{\sigma_{I,0}(u)} \cdots B_{r+1}^{\sigma_{I,r+1}(u)},\]
where the $\sigma_{I,j}(u)$'s are as in equation (\ref{2E1}). By Proposition \ref{2P2}, this is equivalent to
$$\displaystyle  |\Delta| \leq s! \left(B_0^{a_{I,0}} \cdots B_{r+1}^{a_{I,r+1}}\right)^{u \mathcal{H}_I(u)} V^{O_{\Bw,d,r}(u^{r})}.$$
Taking logarithms and recalling (\ref{3E2}), this bound becomes
\begin{equation} \label{6E2} \displaystyle \log |\Delta| \leq u \mathcal{H}_I(u) \log (B_0^{a_{I,0}} \cdots B_{r+1}^{a_{I,r+1}}) + s \log s + O_{\Bw,d,r}\left(u^{r} \log V\right). \end{equation}
We want to express everything in terms of $s$. By (\ref{3E2}), (\ref{6E0}), and (\ref{hilbert}), equation (\ref{6E2}) becomes, for some positive $C_{9}(\Bw, d,r)$,
\begin{equation} \label{6E3} \displaystyle \log |\Delta| \leq (r!)^{1/r}\frac{r}{r+1} s^{1 + 1/r} \log W + s \log s + C_{9}(\Bw, d,r) s \log V. \end{equation}
%We now make a remark regarding the size of $W$ and $V$. With both of our choices of $s$, with respect to part (a) and (b) of Theorem \ref{MT3} in (\ref{6E11}) and (\ref{6E9}) respectively, $s^{1/r}$ will be larger than $\log V$. Thus, even though $\log W$ may be significantly smaller than $\log V$, the first term in (\ref{6E3}) will constitute the main term. \\ \\
We proceed to prove the first part of the theorem. Let $\ep > 0$ be as in the theorem, and recall the hypothesis
\[WV^\ep \leq \Q \leq  WV^{2 \ep} .\]
Choose $s$ to be
\begin{equation} \label{6E11} s = \left \lceil C_{10}(\Bw, d, r)\left(\frac{r+1}{\ep r(r!)^{1/r}} \left( 1 + 2 \ep+ \left(\frac{w}{d}\right)^{1/r} \right) \right)^r  \right \rceil + 1, \end{equation}
where $C_{10}(\Bw,d, r)$ is a positive number which will be chosen later. For each prime $p$, write $\lvert \cdot \rvert_p$ for the $p$-adic valuation on $\bQ$, normalized so that $\lvert p \rvert_p = p^{-1}$. For convenience, let us write
\[\P = \{p_1, \cdots, p_t\}\]
and
\[ \mathfrak{P} = (P_1, \cdots, P_t),\]
where $P_i$ is a non-singular point on $X_{p_i}$ for each $i$, $1 \leq i \leq t$. Theorem \ref{T1} gives that
\[-\log \lvert \Delta \rvert_{p_i} > \frac{(r!)^{1/r} r}{r+1} s^{1 + 1/r} \log p_i - \kappa'(r) s \log p_i.\]
Observe that
\begin{equation} \label{edash} -\sum_{i=1}^t \log \lvert \Delta \rvert_{p_i} > \frac{(r!)^{1/r} r}{r+1} s^{1 + 1/r} \log \Q - \kappa'(r)  s \log \Q. \end{equation}
By (\ref{6E3}) and (\ref{edash}), there exists a positive number $C_{11}(\Bw,d, r)$ such that
\begin{equation} \label{6E10} \log \lvert \Delta \rvert + \sum_{i=1}^t \log \lvert \Delta \rvert_{p_i} \leq \frac{r (r!)^{1/r}}{r+1} s^{1 + 1/r} \log \frac{ W}{\Q} + C_{11}(\Bw, d, r) s \log V\Q. \end{equation}
We choose $C_{10}(\Bw,d, r)$ to be $C_{11}(\Bw,d, r)^r$. Note that by (\ref{2E4}), we have
\[a_{I,j} \frac{r+1}{r} \left(\frac{w}{d}\right)^{1/r} \leq \left(\frac{w}{d}\right)^{1/r},\]
for $0 \leq j \leq r+1$, whence
\[\log W \leq \left(\frac{w}{d}\right)^{1/r} \log V.\]
By the assumption that $\Q \geq WV^\ep$, the right hand side of (\ref{6E10}) then satisfies
\[\frac{r (r!)^{1/r}}{r+1} s^{1 + 1/r} \log \frac{ W}{\Q} + C_{11}(\Bw,d, r) s \log V\Q  \leq - \ep \frac{r(r!)^{1/r}}{r+1} s^{1 + \frac{1}{r}} \log V + C_{11}(\Bw, d,r)  s \log V\left(1 + 2 \ep + \left(\frac{w}{d}\right)^{1/r}\right), \]
and upon dividing the right hand side by $s \log V$ we have
\begin{equation} \label{forced to be zero}- \ep \frac{r(r!)^{1/r}}{r+1} s^{\frac{1}{r}} + C_{11}(\Bw, d,r)  \left(1 + 2\ep + \left(\frac{w}{d}\right)^{1/r}\right). \end{equation}
If (\ref{6E11}) is satisfied, then 
\[ \ep \frac{r(r!)^{1/r}}{r+1} s^{\frac{1}{r}} > C_{11}(\Bw, d,r)  \left(1 + 2 \ep + \left(\frac{w}{d}\right)^{1/r}\right), \]
whence (\ref{forced to be zero}) is negative. Therefore, we obtain
\begin{equation} \label{6E10A} \log \lvert \Delta \rvert + \sum_{i=1}^t \log \lvert \Delta \rvert_{p_i} < 0.
\end{equation}
Hence, for $WV^\ep \leq \Q \leq WV^{2 \ep}$ and $s$ satisfying (\ref{6E11}), we have
\[ \Delta = 0. \]
This implies that $X(\bQ; \BB, \mathfrak{P})$ is contained in a hypersurface $Y(P_1, \cdots, P_t)$ satisfying
\[\deg Y = O_{d, \Bw,r} \left(s^{1/r}\right) = O_{d,\Bw, r, \ep}(1),\]
defined by a primitive form $G$. To estimate the height of $G$, we argue as in Lemma \ref{5Lem3}. Let $s' \leq s - 1$ denote the rank of $(M_j(\boldsymbol \xi_l))$. Then, from evaluating all $(s'+1) \times (s'+1)$ sub-determinants by expanding along a row, we see that the height of $G$ is at most 
\[\max |\det \M|\]
where the maximum is taken over all $s' \times s'$ minors of $(M_j(\boldsymbol \xi_l))$. This can be bounded just as in (\ref{6E3}), so by (\ref{6E11}) and (\ref{4EE2}), we obtain
\[ \log \lVert G \rVert = O_{d,\Bw,r,\ep}\left(\log V \Q \right). \]
Further, since the monomials which appear in $G$ with a non-zero coefficient are not leading monomials of $I$, $F$ cannot divide $G$; and thus, $X$ cannot be contained in $Y(\mathfrak{P})$. This completes the proof of the first part of Theorem \ref{MT3}. \\ \\
For the second part, suppose that $X$ is geometrically integral. Set
\begin{equation} \label{6E9} s = \left \lceil \max\left \{ \Q^{-r}W^r(1+ \log V\Q)^{C_{12}(\Bw, d, r) }, (\log \Q V)^r \right \} \right \rceil + 1, \end{equation}
where $C_{12}(\Bw, d, r)$ is a number which depends on $d$, $\Bw,$ and $r$, and will be specified later; see (\ref{kappa}). By (\ref{6E9}), it follows that
\begin{equation} \label{degu} u = O_{\Bw, d, r} \left( (\Q^{-1} W + 1) \log V\Q \right) . \end{equation}
We now consider the two cases given by Lemma \ref{5Lem3}. If case (a) holds, we can produce a hypersurface $Y$ of degree $d$, distinct from $X$, which contains $X(\bQ; \BB, \mathfrak{P})$. This is sufficient for the theorem. Thus, it remains to treat the case when $\pi_X = O_{\Bw, d, r}(1 + \log V)$. In this case, we have will have two separate divisors of $\Delta$ to estimate; one coming from the prime factors of $\Q$, and one coming from primes which do not divide $\Q \pi_X$. \\ \\
We now estimate the contribution coming from primes which are co-prime to $\Q \pi_X$. For each prime $p$ such that $X_p$ is geometrically integral, by Lemma \ref{6Lem1} we have
\[ - \log \lvert \Delta \rvert_p \geq \frac{(r!)^{1/r} r}{r+1} s^{1 + 1/r} \frac{\log p}{n_p^{1/r}} - \kappa(d,r) s \log p. \]
We write the sum over the primes $p$ for which $p \nmid \Q \pi_X, p \leq s^{1/r}$ as $\displaystyle \sideset{}{^\ast}\sum_{p \leq s^{1/r}}$. By Lemma \ref{5Lem3}, we have
\begin{equation} \label{6E4} \sum_{p | \Q\pi_X} \frac{\log p}{p} = \log(1 + \log V\Q) + O_{\Bw, d, r}(1).\end{equation}
Then, by applying Lemma \ref{6Lem2} and the prime number theorem, we have, for some positive numbers $C_{13}(\Bw, d, r), C_{14}(\Bw, d, r)$,
\begin{align*} -\sideset{}{^\ast} \sum_{p \leq s^{1/r}} \log \lvert \Delta \rvert_p &
 \geq \frac{(r!)^{1/r} r}{r+1} s^{1+1/r} \sideset{}{^\ast} \sum_{p \leq s^{1/r}} \frac{\log p }{n_p^{1/r}} - \kappa(d,r) s \sum_{p \leq s^{1/r}} \log p \\
& \geq \frac{(r!)^{1/r} r}{r+1} s^{1+1/r} \sideset{}{^\ast} \sum_{p \leq s^{1/r}} \frac{\log p}{p} - C_{13} (\Bw, d, r) s^{1 + 1/r} \\ 
& \geq \frac{(r!)^{1/r} }{r+1} s^{1+1/r}  \left(\log s - r \sum_{p | \Q \pi_X} \frac{\log p}{p}\right) - C_{14} (\Bw, d, r) s^{1 + 1/r} \\
& \geq \frac{(r!)^{1/r} }{r+1} s^{1+1/r} \left( \log s - O_{\Bw, d, r}(\log(1 + \log V\Q)) \right) - C_{14}(\Bw, d, r)  s^{1 + 1/r}. 
\end{align*}
We invoke the bound from equation (\ref{6E3}) and obtain the inequality
\begin{equation} \label{6E6} \log \lvert \Delta \rvert + \sum_{i=1}^t \log \lvert \Delta \rvert_{p_i} + \sideset{}{^\ast} \sum_{p \leq s^{1/r}} \log \lvert \Delta \rvert_p \leq \frac{(r!)^{1/r}}{r+1} s^{1 + 1/r} \log\left[ \frac{ W^{r}}{\Q^r s} \right] + C_{15}(\Bw, d, r) \left(s^{1+1/r} (\log (1 + \log V\Q)) + s\log V\Q \right),\end{equation}
where $C_{15}(\Bw, d, r)$ is a positive number which depends on $d$ and $r$. Note that 
\[\log V\Q  \ll_{\Bw, d, r} s^{1/r}\]
by (\ref{6E9}). We may thus choose a positive number $C_{12}(\Bw, d, r)$ such that 
\begin{equation} \label{kappa} C_{15}(\Bw, d, r) \left(s^{1 + 1/r} + s\log V\Q \right) < \frac{(r!)^{1/r}}{r+1} s^{1 + 1/r}  \C_{12}(\Bw, d, r) \log  (1 + \log V \Q).\end{equation}
Then, equation (\ref{6E6}) becomes
\begin{equation} \label{6E7} \log \lvert \Delta \rvert + \sum_{i=1}^t \log \lvert \Delta \rvert_{p_i} + \sideset{}{^\ast} \sum_{p \leq s^{1/r}} \log \lvert \Delta \rvert_p \leq \frac{(r!)^{1/r}}{r+1} s^{1 + 1/r} \log\left[ \frac{(1 + \log V\Q)^{C_{12}(\Bw,d,r)}W^{r} }{\Q^r s} \right].
\end{equation}
Hence,
\begin{equation} \label{6E8} \Delta = 0\end{equation}
whenever
\[ s > \max \left \{\Q^{-r} W^{r} (1 + \log V\Q)^{C_{12}(\Bw, d, r)}, (\log V\Q)^r \right \}.\]
By our choice of $s$ and $C_{12}(\Bw, d, r)$, this is satisfied. \\ \\
When $s$ is of this size, any set of $s$ $(r+2)$-tuples $\boldsymbol \xi_1, \cdots, \boldsymbol \xi_s \in X(\bQ; \BB; P_1, \cdots, P_t)$ satisfies
$$\displaystyle \Delta = 0,$$
so $(M_j(\boldsymbol \xi_l))$ has rank less than $s$. This implies that $(M_j(\boldsymbol \xi_l))$ has a non-trivial kernel, whence we may find an auxiliary form $G$ of degree $u$ defining a hypersurface $Y(P_1, \cdots, P_t)$ such that 
\[X(\bQ; \BB, P_1, \cdots, P_{t}) \subset Y(P_1, \cdots, P_t).\] 
Further, since the monomials which appear in $G$ with non-zero coefficient are not leading monomials of $I$, it follows that $F$ cannot divide $G$. Since $X$ is geometrically integral, the hypersurface $Y(P_1, \cdots, P_t)$ satisfies the conditions of the theorem. This completes the proof of Theorem \ref{MT3}.

%\begin{remark} Note that in equation \ref{6E5} we required 
%$$s > \max \left \{\mathfrak{C}q^{-r} W^{(r+1) \left(\frac{w}{d}\right)^{1/r} } (1 + \log Vq)^r, (\log Vq)^r \right \}.$$
%Earlier in the proof, we had required a seemingly weaker condition that $s \geq (\log Vq)^r$. If $q^r$ and $W^{(r+1)\left(\frac{w}{d}\right)^{1/r}}$ are roughly the same size, then in fact $s$ is of magnitude about the same as $(\log Vq)^r$. This choice will be required in Section 7. 
%\end{remark}

%%%%%%%%%%%%%%%%%%%%%%%%%%%%%%%%%%%%%%%%%
\section{Preliminaries for dealing with binary forms}
\label{S8}
%%%%%%%%%%%%%%%%%%%%%%%%%%%%%%%%%%%%%%%%%

In this section and the next, we use our results from previous sections to prove Theorem \ref{MT1}. Suppose we have a binary form $F(x,y)$ of degree $D$ with integer coefficients. Notice that if $k \geq d/2$, Theorem \ref{MT1} follows from the work of Greaves \cite{<Gre>}. Hence, we may suppose that $k \geq 2$ is an integer which satisfies
\begin{equation} \label{7E2}\frac{7}{18} < \frac{k}{d} < \frac{1}{2}.\end{equation} 
We turn our attention to the following central object
\begin{equation} N_{F,k}(B) = \# \{(x,y) \in \bZ^2 : 1 \leq x, y \leq B, F(x,y) \text{ is } k\text{-free}\}.\end{equation}
We assume that for all primes $p$, there exists a pair of positive integers $(a,b)$, such that $p^k$ does not divide $F(a,b)$. Our strategy will be to show that subject to (\ref{7E2}), we have $N_{F,k}(B) = C_{F,k}B^2 + O(B^2(\log B)^{-\delta})$, where $C_{F,k}$ is as in (\ref{E1}). This would show that $F$ takes on $k$-free values infinitely often. We also note the following observation, which follows easily from the definition of the Mobius function:

$$\displaystyle \sum_{b^k | F(x, y)} \mu(b) = 
\begin{cases}
 1, & \text{if }F(x, y)\text{ is } k\text{-free,} \\
 0, & \text{otherwise.}
\end{cases}$$
For any $\xi > 0$, we write 
$$\displaystyle M_1(B) = \#\{(x,y) \in \bZ^2 : 1 \leq x, y \leq B : p^k | F(x,y) \Rightarrow p > \xi\},$$
$$M_2(B) = \#\left\{(x,y) \in \bZ^2, 1 \leq x, y \leq B : p^k | F(x,y) \Rightarrow p > \xi, \exists p \in \left(\xi, \frac{B^{2}}{\log B}\right] \text{ s.t. } p^k | F(x,y)\right\},$$
and
\[ M_3(B)  = \#\left\{(x,y) \in \bZ^2, 1 \leq x, y \leq B : \exists p > \frac{B^2}{\log B}, v \in \bZ \text{ s.t. } F(x,y) = vp^k \right\}.\]
Note that by their definitions we have
\[ M_1(B) - M_2(B) - M_3(B) \leq N_{F,k}(B) \leq M_1(B),\]
so it suffices to show that $M_1(B)$ dominates the other two terms. Write 
$$\displaystyle N(b, B) = \# \{(x,y) \in \bZ^2 : |x|, |y| \leq B, b^k | F(x,y)\}.$$ 
We have that
\begin{align*} \displaystyle M_1(B) & = \sum_{\substack{b \in \bN \\ p | b \Rightarrow p \leq \xi}} \mu(b) N(b,B) \\
 & =  \sum_{\substack{b \in \bN \\ p | b \Rightarrow p \leq \xi}} \mu(b)\rho_F(b^k)\left \{\frac{B^2}{b^{2k}} + O\left(\frac{B}{b^k} + 1 \right) \right\}. \end{align*}
When $b$ is squarefree, we have the bound
\[ b \leq \prod_{p \leq \xi} p = \exp \left(\sum_{p \leq \xi} \log p \right) \leq e^{2 \xi},\]
by Theorem 4 of \cite{<RS>}. It is clear that the function $\rho_F$ is multiplicative. Since $F$ is a binary form, we see that if $F(x,0) \equiv 0 \pmod{p}$, then $a_D x^D \equiv 0 \pmod{p}$, where $a_D$ is the coefficient of $x^D$ in $F$. There can only be finitely many primes $p$ such that $p | a_D$, and for all other primes we must have $x \equiv 0 \pmod{p}$. In other words, for all but finitely many primes, $0$ is the only solution to $F(x,0) \equiv 0 \pmod{p}$. A similar argument applies for solutions of the form $(0,y)$. Now, suppose that $(x,y)$ is a solution such that $x,y \not\equiv 0 \pmod{p}$. Then,
 \[F(x,y) \equiv y^D F(x/y, 1) \equiv 0 \pmod{p},\]
 and since $y \not \equiv 0 \pmod{p}$, it follows that this solution arises from a zero of $F(\gamma,1)$ over the field of $p$ elements. However, there can be at most $D$ roots to this polynomial, which implies that $\rho_F(p) \ll p$. For $\rho_F(p^k)$, we refer the reader to Lemma 1 of \cite{<Fil>} for the proof of the bound $\rho_F(p^k) \ll p^{2k-2}$. Hence, for any $\ep > 0$ and $b$ square-free, we have $\rho_F(b^k) \ll_\ep b^{2k - 2 + \ep}$. For $k \geq 2$, we have 
$$\displaystyle M_1(B) = B^2 \prod_{p \leq \xi} \left(1 - \frac{\rho_F(p^k)}{p^{2k}}\right) + O\left(\sum_{b \leq e^{2\xi}} (Bb^{k - 2 + \ep} + b^{2k - 2 + \ep})\right).$$
Note that
$$\displaystyle \prod_{p \leq \xi} \left(1 - \frac{\rho_F(p^k)}{p^{2k}}\right)$$
is a partial product of an absolutely convergent product, $C_{F,k}$, and is therefore positive. \\ \\
By setting $\displaystyle \xi = \frac{1}{2k} \log B$, we see that
$$\displaystyle M_1(B) = B^2 \prod_{p \leq \xi} \left(1 - \frac{\rho_F(p^k)}{p^{2k}}\right) + O\left(B^{2 - \frac{1}{k} + \ep}\right).$$
We now consider $M_2(B)$. We refer the reader to Lemma 2 in Greaves \cite{<Gre>}, where he obtained the bound
$$\displaystyle M_2(B) = O\left(B^2(\log B)^{-1}\right),$$
for $k \geq d/2$ and 
\[M_2(B) = O\left(B^2(\log B)^{-1/2}\right)\]
for $k = 2, d = 6$. Helfgott, in \cite{<Hel>}, obtained the error term
\[M_2(B) = O\left(B^2(\log B)^{-\delta}\right)\]
for $\delta = 0.7034\cdots$. We note that the argument in \cite{<Gre>} deals with essentially one prime at a time, so it simultaneously deals with all numbers $z$ divisible by some prime $p$ in the interval $(\xi, B^2(\log B)^{-1}]$. An important feature of Greaves' estimate which is not present in the work of any subsequent author, except Hooley \cite{<Hoo2>} \cite{<Hoo3>}, is that his estimate for $M_2(B)$ is \emph{independent} of any relationship between $k$ and $d$. All further estimates obtained by other authors require a relationship between $k$ and $d$ of the form $k \geq \upsilon_1 d + \upsilon_2$, where $0 < \upsilon_1 \leq 1/2$ and $\upsilon_2 \in \bR$. \\ \\
To complete the proof of the theorem, it will be enough to show that 
$$\displaystyle M_3(B) \ll B^{2 - \eta}$$
for some $\eta > 0$, which will be the focus of the next section.

%%%%%%%%%%%%%%%%%%%%%%%%%%%%%%%%%%%%%%%%
\section{Application of the determinant method and the proof of Theorem \ref{MT1}}
\label{S9}
%%%%%%%%%%%%%%%%%%%%%%%%%%%%%%%%%%%%%%%%

We estimate the remaining term $M_3(B)$ via the generalization of Salberger's global determinant method (see \cite{<S2>}) in the weighted projective case established in earlier sections. The argument given here is specialized for the binary form problem. We denote by 
\[ S_3^{(f)}(B) = \left\{(x,y) \in \bZ^2 : 1 \leq x, y \leq B, \exists p > \frac{B^2}{\log B}, v \in \bZ \text{ s.t. } f(x,y) = vp^k \right\}\]
for some irreducible factor $f$ of $F$. Further, write $M_3^{(f)}(B) = \# S_3^{(f)}(B)$. Since $F$ has non-zero discriminant, it follows that
\[ M_3(B) \leq \sum_{f | F, f \text{ irreducible}} M_3^{(f)}(B).\]
Let us fix an irreducible factor $f(x,y) \in \bZ[x,y]$ of $F$, such that $f$ has maximal degree, and write $d = \deg f$. Note that if $p^k > d \lVert f \rVert B^d$, then $p^k > |f(x,y)|$ for all $(x,y) \in [1,B]^2 \cap \bZ^2$. Therefore, $p^k$ cannot divide $f(x,y)$ unless $(x,y) = (0,0)$. Hence, we may assume that $p \leq \left(d \lVert f \rVert\right)^{1/k} B^{d/k}.$ Thus, the relevant range of primes left to consider are
\[\frac{B^2}{\log B} < p \leq d \lVert f \rVert  B^{\frac{d}{k}}. \]
Following Browning in \cite{<B2>} , we partition the above range into dyadic intervals of the form $(H/2, H]$ where 
\[B^2/\log B  \ll H \ll d \lVert f \rVert B^{\frac{d}{k}}.\]
Now write 
\begin{equation} \label{8E0}R(f; H, B) = \# \{(x,y,v,z) \in \bZ^4 : f(x,y) = vz^k, (x,y) \in S_3^{(f)}(B), \end{equation}
\[\gcd(x,y) =1, H/2 < z \leq H, v \ll B^d/H^k, z \text{ prime}, v \ne 0\}.\]
Write $H = B^\beta$, so $B^d/H^k = B^{d-k\beta}$. Summing over these dyadic intervals, we then obtain:
\begin{equation} \label{Dyadic} M_3^{(f)}(B) \ll \log B \sup_{2 - \frac{\log \log B}{\log B} < \beta \leq \frac{d}{k} + \frac{\log(d\lVert f \rVert)}{\log B}} R(f; B^\beta, B). \end{equation}
Therefore, it suffices to examine the maximum size of a single $R(f; B^\beta, B)$, as in \cite{<B2>}. Diverging from Browning's argument, we directly estimate $R(f; B^\beta, B)$ instead of passing to the single variable case. We are then left to count the number of integral solutions to
\begin{equation} \label{8E1} \F(x,y,v,z) = f(x,y) - vz^k = 0 \end{equation}
where $(x,y,v,z)$ is subject to the constraints in (\ref{8E0}) with $H = B^\beta$. \\ \\
Let us denote by $X$ the surface given by (\ref{8E1}). We consider possible singular points in $X(\bQ; \BB)$. First, note that our ambient space $\bP(1,1,2,d-2k)$ has singularities at $(0,0,1,0)$ and $(0,0,0,1)$. However, these points do not line in $X(\bQ;\BB)$ since we are only counting points whose first two coordinates are co-prime. Next recall that a point $\Bz = (x_0, y_0, v_0, z_0)$ on $X$ is singular if 
\[\frac{\partial \F}{\partial x} (\Bz) = \frac{\partial \F}{\partial y} (\Bz) = \frac{\partial \F}{\partial v} (\Bz) = \frac{\partial \F}{\partial z} (\Bz) = 0.\]
Suppose that $\displaystyle \frac{\partial \F}{\partial x} (\Bz) = \frac{\partial \F}{\partial y} (\Bz) = 0$, with $\Bz \ne \mathbf{0}$. Then, by Euler's formula, we have
\begin{align*} 0 & = \left(x_0 \frac{\partial \F}{\partial x} (\Bz) + y_0 \frac{\partial \F}{\partial y} (\Bz) \right) \\ 
& = \left(x_0 \frac{\partial f}{\partial x} (\Bz) + y_0 \frac{\partial f}{\partial y} (\Bz) \right) \\ 
& = d f(x_0, y_0).
\end{align*}
Since $f$ is irreducible over $\bQ$, it has no integral zeroes except $(0,0)$. Therefore, we see that all points in $X(\bQ; \BB)$ are non-singular, since it only counts those points where the first two coordinates are co-prime. \\ \\
Let 
\begin{equation} \label{XBeta} X^\beta(\bQ; \BB) = \{\Bx \in X : \Bx \text{ satisfies (\ref{8E0}) }  \}. \end{equation}In view of Proposition \ref{2P2}, we need to compute the constants $a_x, a_y, a_v, a_z$ with respect to the ideal $I = \langle  \F \rangle$. By (\ref{8E0}), we have
\[B_x = B_y = B, B_v = B^{d - k\beta}, B_z = B^\beta.\]
Note that with respect to reverse lexicographic ordering, the monomial $vz^k$ is maximal in $\F$. Hence, it follows that
\[a_x = a_y = \frac{d - 0}{3d},\]
\[a_v = \frac{d - (d-2k)}{3d(d-2k)} = \frac{2k}{3d(d-2k)},\]
and
\[a_z = \frac{d - 2(k)}{3d(2)} = \frac{d-2k}{6d}.\]
Thus, we have
\[B_x^{a_x} B_y^{a_y} B_v^{a_v} B_z^{a_z} = B^{\frac{1}{3} \left(2 + \frac{2k(d-k\beta)}{d(d-2k)} + \frac{\beta(d-2k)}{2d} \right)}.\]
Next, note that
\[1 + \frac{2k(d-k\beta)}{d(d-2k)} + \frac{\beta(d-2k)}{2d} = \frac{d-k\beta}{d-2k} + \frac{\beta}{2},\]
whence it follows
\[\left(B_x^{a_x} B_y^{a_y} B_v^{a_v} B_z^{a_z}\right)^{\frac{3}{2} \left(\frac{2(d-2k)}{d}\right)^{1/2}} = \left(B^{1 + \frac{d-k\beta}{d-2k} + \frac{\beta}{2}}\right)^{ \frac{1}{2} \left(  \frac{2(d-2k)}{d}\right)^{1/2}}.\]
Let us write 
\begin{equation} \label{exp} \psi = \frac{1}{2} \left(1 + \frac{d-k\beta}{d-2k} + \frac{\beta}{2} \right) \left( \frac{2(d-2k)}{d}\right)^{1/2}.\end{equation}
Observe that $B^\psi$ corresponds to $W$ in Theorem \ref{MT3}. \\ \\
It is clear that $X$ is geometrically integral. Hence, by Theorem \ref{MT3}, there exists a surface $Y(\emptyset) \subset \bP(1,1,d-2k,2)$ not containing $X$ such that 
\begin{equation} \label{yempty} \deg Y(\emptyset) = O_{d, \ep}\left(B^{\psi+\ep}\right)\end{equation}
and 
\[X^\beta(\bQ; \BB) \subset X(\bQ;\BB) \subset Y(\emptyset).\]
We will now show that, in fact, $X_p$ is geometrically integral except for those primes $p$ which divide the coefficients of $x^d$ and $y^d$ in $f(x,y)$. Suppose that 
\[\F(x,y,v,z) = f(x,y) - vz^k\]
admits a factorization into two weighted forms $\F_1, \F_2$ over the algebraic closure of $\ol{\bF_p}$, where $p$ does not divide the coefficient of $x^d$ nor $y^d$ in $f(x,y)$. By Lemma 8 in Chapter 2 of \cite{<CLO>}, it follows that the leading monomial of $\F$ is equal to the product of the leading monomials of $\F_1, \F_2$. Thus, under our ordering $>$, where $vz^k$ is the leading monomial of $\F$, this implies that $\F_1, \F_2$ must take the forms
\[\F_1(x,y,v,z) = a_0 vz^l + \G_1(x,y,z),\]
\[\F_2(x,y,v,z) = b_0 z^{k-l} + \G_2(x,y,z)\]
for some non-negative integer $l \leq k$, since $\F_1, \F_2$ are both weighted homogeneous with respect to $(1,1,d-2k, 2)$. By considering different orderings which order $x$ and $y$ respectively as the highest and applying Lemma 8 in Chapter 2 of \cite{<CLO>}, we see that 
\[\F_1(x,y,v,z) =a_0 vz^l + a_1 x^{d-2k+2l} + a_2 y^{d-2k+2l} + \G_1'(x,y,z),\]
\[\F_2(x,y,v,z) = b_0 z^{k-l} + b_1 x^{2k-2l} + b_2 y^{2k-2l} + \G_2'(x,y,z)\]
where $a_1, a_2, b_1, b_2$ are non-zero in $\ol{\bF_p}$. The terms 
\[a_0 b_1 x^{2k-2l} vz^l, a_0 b_2 y^{2k-2l} vz^l, b_0 a_1 x^{d-2k+2l} z^{k-l}, b_0 a_2 y^{d-2k+2l} z^{k-l}\]
must appear in $\F = \F_1 \F_2$ with non-zero coefficient, which is plainly not the case. This contradiction implies that $X_p$ is geometrically integral over $\bF_p$ whenever $p$ does not divide the coefficients of $x^d$ and $y^d$. \\ \\ 
Recall the definition of $\pi_X$ (Definition \ref{D3}) from Section \ref{S4}. By the preceding argument, it follows that $\pi_X \leq  \lVert f \rVert$. Let $0 < \ep < 1/2$ be a positive number, and let $\{p_1, p_2, \cdots\}$ be the increasing sequence of consecutive primes larger than $\max\{\lVert f \rVert, \log B\}$ for which
\begin{equation} \label{eight-a} p_1 \cdots p_t < B^{\psi + \ep} \leq p_1 \cdots p_{t+1}.\end{equation}
We now give an estimate for $p_{t+1}$. Let
\[\theta(x) = \sum_{p \leq x} \log p, \]
and let us write $\Q_j = p_1 \cdots p_{j}$ for $j = 1,2, \cdots, t+1$, with $\Q_0 = 1$. By the Prime Number Theorem, there exists some absolute constant $C_{16}$ such that
\begin{equation} \label{8E5} p_{t+1} < C_{16} \theta(p_{t+1}) = C_{16} \sum_{p \leq p_{t+1}} \log p, \end{equation}
hence
\begin{align*} 
p_{t+1} - C_{16} \log p_{t+1} & \ll  \sum_{p \leq \max\{\lVert f \rVert, \log B\}} \log p + \sum_{\max\{\lVert f \rVert, \log B\} < p \leq p_{t}} \log p \\ 
& \leq \theta(\log B) + \sum_{p | \pi_X} \log p + \sum_{j=1}^{t} \log p \\
& = \theta(\log B) + \sum_{p | \pi_X} \log p + \log \Q_{t} \\
&\ll \log B + \lVert f \rVert,  \end{align*}
since we know that $p_{t+1} > \max\{\log B, \lVert f \rVert\}$ and therefore we can, by choosing $B$ sufficiently large, make sure that $C_{16} \log p_{t+1} < \frac{1}{2} p_{t+1}.$
Thus, we have
\begin{equation} \label{qt} \Q_{t+1} = O( B^{\psi + \ep} \log B).
\end{equation}
Since the partial derivative
\[\frac{\partial \F}{\partial v} = z^k\]
is only divisible by primes $\gg B^2(\log B)^{-1}$, (\ref{8E5}) implies 
that there is no point $\Bx \in X^\beta(\bQ; \BB)$ which specializes to a singular point on $X_{p_j}$ for $j=1, \cdots, t+1$. Hence, every $\Bx \in X^\beta(\bQ; \BB)$ reduces to a non-singular point on $X_{p_j}$ for every prime $j= 1, \cdots, t+1$. \\ \\
Our goal now is to construct a set of exceptional points $\E$ and a collection of curves $\Gamma$ which cover $X^\beta(\bQ; \BB)$. Consider an irreducible component $\D(\emptyset)$ of $X \cap Y(\emptyset)$. For each point $\Bx \in \D(\emptyset) \cap X^\beta(\bQ; \BB)$, let $P_1(\Bx) = P_1$ be the $\bF_{p_1}$-point on $X_{p_1}$ such that $\Bx \equiv P_1 \pmod{p_1}$. By Theorem \ref{MT3}, there exists a surface $Y(P_1)$ which contains $X^\beta(\bQ; \BB, P_1)$. Thus, there exists an irreducible component $\D_{\Bx}(P_1)$ of $X \cap Y(P_1)$ which contains $\Bx$. If $\D(\emptyset) \ne \D_\Bx(P_1)$, then put $\Bx$ in a set $Z(P_1)$. Repeat this process for each irreducible component $\D$ of $X \cap Y(\emptyset)$, to obtain sets $Z(P_1)$ for each $P_1 \in X_{p_1}$. Note that a surface in $\bP(1,1,2,d-2k)$ of weighted degree $d$ is the quotient of a certain action of a surface of degree $d$ in the straight projective space $\bP^3$, therefore B\'{e}zout's Theorem for straight projective spaces provides an upper bound for the cardinality of the sets $Z(P_1)$. Theorem \ref{MT3} then shows that for each $P_1 \in X_{p_1}$, we have
\[\# Z(P_1) \ll_d  \left(p_1^{-1}B^\psi + \log Bp_1 \right)\left( B^\psi + \log B \right)(\log B)^2.\]
Write 
\[Z(p_1) = \bigcup_{P_1 \in X_{p_1}} Z(P_1).\]
By Lang and Weil's theorem, we have $\# X_{p_1} = O_d(p_1^2)$, where it follows that
\[\# Z(p_1) = O_d \left(p_1^2 \left(p_1^{-1}B^{2\psi} + \log Bp_1 \right)(\log B)^2 \right) = O_d\left(B^{2\psi}(\log B)^5 \right).\]
What remains are irreducible components $\C$ of $X \cap Y(\emptyset)$ which are also irreducible components of $X \cap Y(P_1)$ for some $P_1 \in X_{p_1}$. Call this collection of curves $\Gamma^{(1)}$. For each surface $Y(P_1)$, suppose that $G_{P_1}$ is a primitive form which defines $Y(P_1)$. Then, from (\ref{8E1}) we see that we can substitute $v = f(x,y)/z^k$ into $G_{P_1}$ to obtain
\begin{equation} \label{curve sub} G_{P_1}(x,y,v,z) = G_{P_1} \left(x,y,\frac{f(x,y)}{z^k}, z\right).\end{equation}

If $G_{P_1}(x,y,v,z)$ has a $v$ term, then we may replace the $v$'s with $f(x,y)/z^k$ to obtain a form over $\bP(1,1,2)$. If $G$ does not have a term containing $v$, then no substitution is necessary and we again obtain a form over $\bP(1,1,2)$. Since $G_{P_1}$ is weighted homogeneous with respect to $(1,1,d-2k,2)$, it follows that each monomial that appears in $G_{P_1}$ with a non-zero coefficient has the same weighted degree $l$ with respect to the weight vector $(1,1,d-2k, 2)$. Consider a monomial $x^{\alpha_1}y^{\alpha_2}v^{\alpha_3} z^{\alpha_4}$ that appears in $G_{P_1}$ with non-zero coefficient. After the substitution, we obtain
\[ x^{\alpha_1} y^{\alpha_2} \left(\frac{f(x,y)}{z^k}\right)^{\alpha_3} z^{\alpha_4}.\]
Expanding $f(x,y)$ and recalling that $f$ is a binary form of degree $d$, it follows that each monomial which appears in the expansion $f(x,y)^{\alpha_3}$ has degree $d \alpha_3$.  Now, we multiply by a large power of $z$, say $z^L$, so that 
\[z^L G_{P_1} \left(x,y,\frac{f(x,y)}{z^k}, z\right)\]
is a polynomial in $x,y,z$. Each monomial that appears in $z^L x^{\alpha_1} y^{\alpha_2} (f(x,y)/z^k)^{\alpha_3} z^{\alpha_4}$ has weighted degree
\[2L + \alpha_1 + \alpha_2 + d\alpha_3 - 2k \alpha_3 + 2 \alpha_4 = 2L + l,\]
so $z^L G_{P_1}(x,y, f(x,y)/z^k, z)$ is a polynomial over $\bP(1,1,2)$. Further, if we choose $L$ to be minimal, then $L \leq kl$. We call the new polynomial $\G_{P_1}(x,y,z)$. It is now clear that the degree of $\G_{P_1}$ is at most $2kl + l = l(2k+1)$, and thus Theorem \ref{MT3} implies
\begin{equation}\label{curve degree bd} \deg \C = \deg \G_{P_1} = O_d\left(\left(p_1^{-1}B^{\psi} \log B + \log Bp_1 \right)\right)  \end{equation}
for each $\C \in \Gamma^{(1)}$. Observe that $\Gamma^{(1)}$ is a collection of irreducible components of $X \cap Y(\emptyset)$, hence
\[\# \Gamma^{(1)} = O_d\left(B^{\psi}\log B \right).\]
We have thus obtained a relatively small set of points $Z(p_1)$ and a collection of curves $\Gamma^{(1)}$ which together cover $X^\beta(\bQ; \BB)$. Moreover, the curves in $\Gamma^{(1)}$ now have degrees bounded as in (\ref{curve degree bd}) and the number of curves in $\Gamma^{(1)}$ is bounded above by the degree of $Y(\emptyset)$. We can continue this process to continue to separate points in $X^\beta(\bQ;\BB)$ into an exceptional set or onto a curve of relatively small degree. \\ \\
Suppose we have obtained $Z(\Q_i)$ for $1 \leq i \leq j$ up to some positive integer $j$. In particular, $Z(\Q_i)$ is the set of points $\Bx \in X^\beta(\bQ;\BB)$ such that $\Bx \not \in Z(\Q_{i-1})$ and $\D(\emptyset) \ne \D_{\Bx}(P_1, \cdots, P_i)$. Notice that
\[\# Z(\Q_i) = O_d \left(\Q_i^2 (p_i^{-1} \Q_{i-1}^{-2} B^{2\psi} + \log B \Q_i) (\log B)^2 \right).\]
Similarly, suppose we have obtained $\Gamma^{(i)}$, $1 \leq i \leq j$, where $\Gamma^{(i)}$ is the set of curves $\C$ of degree
\[O_d\left(\Q_i^{-1} B^\psi \log B + \log B\Q_i \right),\]
such that $\C \in \Gamma^{(i-1)}$ and 
\[\C = \D(\emptyset) = \D_{\Bx}(P_1, \cdots, P_i)\]
for some $(P_1, \cdots, P_i)$. Observe that we have
\[X^\beta(\bQ; \BB) \subset \bigcup_{\C \in \Gamma^{(i)}} \C \cup Z(\Q_i).\]
We now construct $Z(\Q_{j+1})$ given $Z(\bQ_j)$. Consider an irreducible curve $\C \in \Gamma^{(j)}$. For each point $\Bx \in \C \cap (X^\beta(\bQ; \BB) \setminus Z(\Q_j))$, we have
\[\D(\emptyset) = \D_\Bx(P_1) = \D_\Bx(P_1,P_2) = \cdots = \D_\Bx(P_1, \cdots, P_j) = \C.\]
There exists a point $P_{j+1} = P_{j+1}(\Bx) \in X_{p_{j+1}}$ such that $\Bx \equiv P_{j+1} \pmod{p_{j+1}}$. Hence, by Theorem \ref{MT3}, there exists a surface $Y(P_1, \cdots, P_{j+1})$ such that $\Bx \in X \cap Y(P_1, \cdots, P_{j+1})$, and
\[\deg Y(P_1, \cdots, P_{j+1}) = O_d \left(\Q_{j+1}^{-1} B^{\psi + \ep} + \log B\Q_{j+1}  \right).\]
Set $\D_\Bx(P_1, \cdots, P_{j+1})$ to be an irreducible component of $X \cap Y(P_1, \cdots, P_{j+1})$ which contains $\Bx$. Put $\Bx$ in the set $Z(P_1, \cdots, P_{j+1})$ if 
\[\D_\Bx(P_1, \cdots, P_j) \ne \D_\Bx(P_1, \cdots, P_{j+1}),\]
then repeat this process for every point $\Bx \in \C \cap (X^\beta(\bQ; \BB) \setminus Z(\Q_j))$ and for every curve in $\Gamma^{(j)}$ to obtain our sets $Z(P_1, \cdots, P_{j+1})$ for $P_i \in X_{p_i}, i = 1, \cdots, j+1$. By B\'{e}zout's theorem, we have

\begin{align} \label{ZP} \ \# Z(P_1, \cdots, P_{j+1}) & = O_d\left(\deg Y(P_1, \cdots, P_j) \deg Y(P_1, \cdots, P_{j+1} \right) \\ 
& = O_d\left(\Q_j^{-1} \Q_{j+1}^{-1} B^{2 \psi} + (\Q_j^{-1} + \Q_{j+1}^{-1}) B^\psi \log B \Q_{j+1} + \log^2 B \Q_{j+1} \right) \notag \\
& = O_d\left(\Q_{j+1}^{-2} B^{2\psi} \log B + \log^2 B \Q_{j+1} \right) \notag
\end{align}
Write $Z(\Q_{j+1})$ as
\[Z(\Q_{j+1}) = \bigcup_{\substack{P_i \in X_{p_i} \\ 1 \leq i \leq j+1}} Z(P_1, \cdots, P_{j+1}).\]
By Lemma \ref{6Lem2}, we have
\[\# X_{p_j} = p_j^2 + O(d^2 p_j^{3/2}) + O_d(p)\]
for $j = 1, \cdots, t+1$. We write this as
\[\# X_{p_j}/p_j^2 = 1 + O(d^2 p_j^{-1/2}) + O_d(p^{-1}). \]
Therefore, for some number $C_{17}(d) > 0$ depending on $d$, we have
\[\prod_{i=1}^{j+1} \frac{\# X_{p_i}}{p_i^2} \leq \left(\prod_{i=1}^{j+1} \left(1 + p_i^{-1/2}\right) \right)^{C_{17}(d)},\]
hence
\[\prod_{i=1}^{j+1} \# X_{p_i} \leq \Q_{j+1}^2 \left( \prod_{i=1}^{j+1} \left(1 + p_i^{-1/2}\right) \right)^{C_{17}(d)}.\]
Since $\Q_{t+1} = p_1 \cdots p_{t+1} \ll B^{\psi+\ep} \log B$ and $p_i \geq \log B$, there exists a positive number $C_{18}(d)$ such that 
\begin{equation} \label{8E12} t \leq \frac{C_{18}(d) \log B}{\log \log B}.\end{equation}
We now use the inequality
\[1 + \upsilon \leq e^\upsilon\]
which is valid for all $\upsilon \geq 0$, to obtain
\[\prod_{i=1}^{j+1} (1 + p_i^{-1/2}) \leq \prod_{i=1}^{j+1} \exp\left(p_i^{-1/2} \right).\]
Noting that $p_i \geq \log B$ for $i = 1, \cdots, j+1$, it follows that
\[ \prod_{i=1}^{j+1} (1 + p_i^{-1/2}) \leq \exp\left((j+1) (\log B)^{-1/2} \right).\]
Hence, by (\ref{8E12}), we have
\[ \prod_{i=1}^{j+1} (1 + p_i^{-1/2}) \leq \exp\left( \frac{C_{18}(d)(\log B)^{1/2}}{\log \log B} \right),\]
so we obtain
\begin{equation} \label{Weilsub} \prod_{i=1}^{j+1} \# X_{p_i} \leq \Q_{j+1}^2 \exp\left(\frac{C_{19}(d) (\log B)^{1/2}}{\log \log B} \right),
\end{equation}
where $C_{19}(d) = C_{17}(d) C_{18}(d)$. By (\ref{ZP}), (\ref{Weilsub}), and Theorem \ref{MT3}, it follows that:
\begin{equation} \label{exception} \#Z(\Q_{j+1}) = O_d \left( \left(B^{2\psi} + \Q_{j+1}^2 \log^2 B\Q_{j+1}\right) \exp\left(\frac{C_{19}(d) (\log B)^{1/2}}{\log \log B} \right)\right).
\end{equation}
We write $\Gamma^{(j+1)}$ to be the set of irreducible curves $\C \in \Gamma^{(j)}$ which are common irreducible components of $X \cap Y(P_1, \cdots, P_j)$ and $X \cap Y(P_1, \cdots, P_{j+1})$. For each curve $\C \in \Gamma^{(j+1)}$, we have
\[\deg \C = O_d\left(\Q_{j+1}^{-1}B^{\psi} \log B + \log B\Q_{j+1} \right).\]
By (\ref{qt}) and (\ref{exception}), we see that
\begin{equation} \label{8E13}  \#Z(\Q_{t+1}) = O_{d,\ep} \left( B^{2 \psi + \ep} (\log B)^2 \exp\left(\frac{C_{19}(d) (\log B)^{1/2}}{\log \log B} \right)  \right) \end{equation}
We write $\Gamma = \Gamma^{(t+1)}$. If $\C \in \Gamma$, then the hypothesis of the first half of Theorem \ref{MT3} applies, whence
\[\deg \C = O_{d,\ep}( 1).\]
We put the sets $Z(\Q_1), \cdots Z(\Q_{t+1}) $ together to form the exceptional set:
\[ \E = \bigcup_{j=1}^{t+1} Z(\Q_j).\]
Then (\ref{8E12}) and (\ref{8E13}) imply that:
\begin{equation} \label{except} \# \E = O_d\left(B^{2\psi + \ep} \exp\left((\log B )^{1/2}/\log \log B \right)(\log B)^3 (\log \log B)^{-1}\right).\end{equation}
We now turn our attention to the set $\Gamma$. Since $\# \Gamma$ does not exceed the number of irreducible components of $X \cap Y(\emptyset)$, it follows from (\ref{yempty}) that
\begin{equation} \label{gam} \#\Gamma = O_{d,\ep}\left(B^{\psi + \ep}\right).\end{equation}
By construction, it follows that
\begin{equation} \label{rf} R(f; B^\beta, B) \leq \# \E + \# \bigcup_{\C \in \Gamma} \C(\bQ; \BB).
\end{equation}
For $\C \in \Gamma$, $\C$ is a component of $Y(P_1, \cdots, P_{t+1})$ for some $(P_1, \cdots, P_{t+1})$. Moreover, since $\Q_{t+1} = p_1 \cdots p_{t+1}$ satisfies the hypothesis of part (a) of Theorem \ref{MT3}, it follows from B\'{e}zout's theorem that
\begin{equation} \label{degc} \deg \C \leq \deg X \cdot \deg Y(P_1, \cdots, P_{t+1}) = O_{d,\ep}(1).\end{equation}
Let $G^\ast$ be a primitive form which defines $Y(P_1, \cdots, P_{t+1})$. By (\ref{eight-a}) and case a) of Theorem \ref{MT3}, we also have
\[\log \lVert G^\ast \rVert = \log (H(Y(P_1, \cdots, P_{t+1}))) = O_{d,\ep}\left(\log B \right).\]
By following the same substitution as in (\ref{curve sub}), we obtain a form over $\bP(1,1,2)$ by substituting (\ref{8E1}) into $G^\ast$. We call the new polynomial $G(x,y,z)$. Observe that 
\[\log \lVert G \rVert = \log \lVert G^\ast \rVert + O_{d} \left(l \log \lVert f \rVert \right).\]
We may now suppose that $B$ is chosen sufficiently large so that $\log \lVert f \rVert < \log B$. Then we obtain
\begin{equation} \label{height-g} \log \lVert G \rVert = O_{d,\ep}\left( \log B \right).
\end{equation}
Note that the curve $\C$ corresponds naturally to a component $\C'$ of the curve $G(x,y,z) = 0$. If $\C'$ is reducible, we consider each irreducible component separately, noting that there are at most $O_{d,\ep}(1)$ components by B\'{e}zout's theorem and (\ref{degc}). Thus, we may consider each irreducible component $\C''$ of $\C'$. There are two situations. First, $\C''$ may be irreducible over $\bQ$, but reducible over $\ol{\bQ}$. In this case, the rational points on $\C''$ are preserved under the all elements of $\Gal(\ol{\bQ}/\bQ)$, but $\C''$ has a conjugate which is also a component of $\C'$, whence $\C''(\bQ)$ corresponds to the rational points in the intersection of two curves each of degree $O_{d,\ep}(1)$; so by B\'{e}zout's theorem, it follows that
\[\# \C''(\bQ) = O_{d,\ep}(1).\]
We suppose now that $\C$ corresponds to a $\bQ$-defined and geometrically integral component of $G$, which we call $\G$. Hence we have
\[\C \leftrightarrow \G(x,y,z) = 0. \]
By Proposition B.7.3 in \cite{<HS>} and (\ref{height-g}), we have
\[\log \lVert \G \rVert = O_{d,\ep}\left(\log B \right).\]
%We proceed to argue as in Lemma \ref{5Lem1}: either the number of integral solutions to $G(x,y,z) = 0$ is at most $l^2$, or its coefficients are bounded by $X^\theta$, where $\theta = l R_3(l)$. Note that $R_3(l) \ll l^2$. The argument is very similar to Lemma \ref{5Lem1}, so we do not repeat it. \\ \\
%Let $\Gamma_0 \subset \Gamma$ denote the set of curves $\C$ corresponding to the first case. Therefore,
%\[\#X^\beta(\bQ; \BB) \cap \bigcup_{\C \in \Gamma_0} \C \leq \sum_{\C \in \Gamma_0} \deg(\C)^2.\]
%Since $\C: G(x,y,z) = 0$ is an irreducible component of $X \cap Y(\emptyset)$, it follows that
%\[\sum_{\C \in \Gamma_0} \deg G \leq (k+1) \deg Y(\emptyset),\]
%which implies
%\[ \sum_{\C \in \Gamma_0} (\deg G)^2 \ll (\deg Y)^2 = O_{d, \ep} \left(B^{2\psi +\ep} \right).\]
We write
\[\G(x,y,z) = G_1(x,y) + zG_2(x,y,z), \]
where $G_1(x,y)$ consists of all monomials in $\G$ which only contains $x$ and $y$. Observe that since $\G \in \bP(1,1,2)$ that $G_1$ is homogeneous in $x$ and $y$. We then consider several situations. \\ \\
Let $\Gamma_1$ denote the set of curves $\C \in \Gamma$ such that $f(x,y)$ and $G_1(x,y)$ are coprime. If $xy = 0$, say $y = 0$, then 
\[f(x,0) = a_d x^d = vz^k.\]
Since we have assumed that $z$ is a prime by (\ref{8E0}), it follows that we must have $z | a_d x^d$. However, since we assumed that $z \gg B^2 (\log B)^{-1}$ and $B > \lVert f \rVert$, this is not possible. It follows that no point with $xy = 0$ can lie in $X^\beta(\bQ; \BB)$. Write $f(x,y) = y^d f(x/y, 1)$ and $G_1(x,y) = y^{\deg(G_1)} G_1(x/y, 1)$. Further, write $h(x) = f(x,1)$ and $g(x) = G_1(x, 1)$. There exist polynomials $a(x), b(x) \in \bZ[x]$ and such that 
\[ a(x)h(x) + b(x)g(x) = \operatorname{Res}(h,g),\]
where $\operatorname{Res}(h,g)$ is the resultant of $h$ and $g$, see \cite{<CLO>}. Homogenizing the equation, we obtain
\[ a'(x,y)f(x,y) + b'(x,y)G_1(x,y) = ny^e,\]
where $e$ is the least positive integer such that the left hand side is a binary form. \\ \\
Since $z | G_1(x,y)$ and $z | f(x,y)$, it follows that $z | ny^e$. However, recall from Section \ref{S8} that $z$ is a prime not smaller than $B^{2}(\log B)^{-1}$, and since $y \in [1, B]$, it follows that $z | n$. The resultant $\operatorname{Res}(h,g)$ is bounded by 
\[|\operatorname{Res}(h,g)| \leq (d + \deg \G+1)! (\lVert f \rVert \cdot \lVert \G \rVert)^{d + \deg \G + 1}. \] 
Hence, the number of prime divisors dividing $n$ of size at least $B^2 (\log B)^{-1}$ is at most
\begin{equation} \label{height} O\left(\frac{\deg \G \log \lVert \G \rVert}{\log B } \right) = O_{d,\ep} \left( 1 \right).\end{equation}
We can now argue as in Greaves \cite{<Gre>}. By (\ref{8E0}), we have that $z$ is in fact a prime. Thus, there are at most $d$ solutions to the congruence
\[f(\omega, 1) \equiv 0 \pmod{z}.\]
By (\ref{BFE}), we have
\[f(x,y) \equiv 0 \pmod{z},\]
and since $xy \not \equiv 0 \pmod{z}$, there exists $\omega \ne 0$ such that $x \equiv \omega y \pmod{z}$. For each such $\omega$, Lemma 1 in Greaves \cite{<Gre>} gives that there are at most
\[\frac{B^2}{z} + O(B) = O(B)\]
such solutions. Thus for each $z$, there are at most $d \cdot O(B) = O_d(B)$ many points in $X^\beta (\bQ; \BB)$ corresponding to a point on a curve $\C \in \Gamma_1$. Since there are $O_\ep(B^\ep)$ choices for $z$ and 
\[O_{d, \ep}\left(B^{\psi + \ep}\right) \]
choices for $\C \in \Gamma_1$, it follows that 
\begin{equation} \label{estimate1} \#X^\beta(\bQ; \BB) \cap \bigcup_{\C \in \Gamma_1} \C = O_{d, \ep}\left( B^{\psi + 1 + \ep}\right). \end{equation}
Next, consider the curves $\Gamma_2 \subset \Gamma$ consisting of those $\C \in \Gamma$ such that $f(x,y), G_1(x,y)$ are not co-prime. As we have chosen $f$ to be irreducible, this implies that $f(x,y)$ divides $G_1(x,y)$. By our choice of $\G$, the degree of $\G$ is at least $d$ and at most $O_{d,\ep}(1)$. We write $l = \deg \C = \deg \G$. We calculate the corresponding quantities $a_x, a_y, a_z$ with respect to the monomial ordering $<$. Suppose that $x^{\alpha_x} y^{\alpha_y} z^{\alpha_z}$ is the leading monomial in $\G$ with respect to reverse lexicographic ordering. In particular, we must have
\[\alpha_x + \alpha_y + 2 \alpha_z = l,\]
since $\G$ is a polynomial over $\bP(1,1,2)$ of weighted degree $l$. Further, we have
\[a_x = \frac{l - \alpha_x}{2l},\]
\[a_y = \frac{l - \alpha_y}{2l},\]
and
\[a_z = \frac{l - 2 \alpha_z}{4l}.\]
Hence,
\begin{align*}B_x^{a_x} B_y^{a_y} B_z^{a_z} & = B^{\frac{1}{4}\left(\frac{(4 + \beta)l - 2\alpha_x - 2\alpha_y - 2 \beta \alpha_z}{l} \right)}\\
& = B^{\frac{1}{4} \left(\frac{(2 + \beta) l + 2\alpha_z(2 - \beta) }{l} \right)}.
\end{align*}
Write
\begin{equation} \label{bigpsi} \Psi = \frac{(2 + \beta) l + 2\alpha_z(2 - \beta) }{l^2}.
\end{equation}
Observe that the $W$ in Theorem \ref{MT3} corresponds to the quantity $B^\Psi$. \\ \\
Now we argue as in \cite{<HB1>}. If $\Bx \in \bP(1,1,2)$ is a singular point on $\C$, then $\Bx$ is a common zero of $\G$ and $\displaystyle \frac{\partial \G}{\partial x}$, hence $\Bx$ lies on the intersection 
\[\C \cap \C',\]
where $\C'$ is the zero-locus of $\displaystyle \frac{\partial \G}{\partial x}$. By B\'{e}zout's theorem, the number of singular points on $\C$ is at most
\begin{equation} \label{singularc} O_{d,\ep}\left(1 \right).\end{equation}
It remains to consider non-singular points on $\C$. Suppose $\Bz \in \C^\beta(\bQ; \BB)$ is non-singular, but reduces to a singular point modulo $p$ for some prime $p$. Then, we must have $p$ divides
\[ \frac{\partial \G}{\partial x} (\Bz), \frac{\partial \G}{\partial y} (\Bz), \frac{\partial \G}{\partial z} (\Bz).\]
However, $\Bz$ is non-singular, so one of the partial derivatives above is non-zero. We may suppose, as we may, that $\displaystyle \frac{\partial \G}{\partial x}(\Bz) \ne 0$. Since
\[\left \lvert \frac{\partial \G}{\partial x}(\Bz) \right \rvert \ll_d l \lVert \G \rVert B^{l - 1},\]
it follows that 
\[\# \left \{p > B^{\Psi + \ep} : p | \frac{\partial \G}{\partial x} (\Bz) \right \} \ll_{d,\ep} \frac{l \log (\lVert \G \rVert B)}{\log 2 + \Psi \log B}. \]
Choose $C_{20}(d,\ep)$ to be a number which depends on $d, \ep$ and gives an upper bound for the inequality above. Now set 
\[n = \left \lceil \frac{C_{20}(d,\ep) l \log (\lVert \G \rVert B)}{\log 2 + \Psi \log B}\right \rceil \ll \log B,\]
where the implied constant is absolute, and
\[q_1 < \cdots < q_n\]
to be the first $n$ primes larger than $B^{\Psi + \ep}$. Then there exists $j$ with $1 \leq j \leq n$ such that $q_j \nmid \displaystyle \frac{\partial \G}{\partial x}(\Bz)$, so $\Bz$ will reduce to a non-singular point on $\C_{q_j}$. By Theorem \ref{MT3} and the theorem of Lang-Weil \cite{<LW>}, there exist 
\[R = O(nq_n) \ll_{d,\ep}  B^{\Psi + \ep} \]
forms $\G_1, \cdots, \G_{R}$ of degree $O_{l,\ep}(1) = O_{d,\ep}(1)$, defining curves $\Y_1, \cdots, \Y_{R}$, such that $\C \not \subset \Y_j$ for $j = 1, \cdots, R$, and
\[\C_{\text{non-singular}}^\beta(\bQ; \BB) \subset \bigcup_{j=1}^R \Y_j.\]
By B\'{e}zout's Theorem, (\ref{singularc}), and Theorem \ref{MT3}, we have the bound
\[ \# \C^\beta(\bQ; \BB) = O_{d,\ep} \left( B^{\frac{2 + \beta}{l} + \frac{2\alpha_z(2 - \beta)}{l^2} + \ep}  \right).\]
Further, we have
\[2 - \frac{\log \log B}{\log B} < \beta \leq \frac{d}{k} + \frac{\log (\lVert f \rVert d)}{\log B}.\] 
If $\beta \geq 2$, then certainly
\[2 - \beta \leq 0,\]
hence
\[2\alpha_z(2 - \beta) \leq 0, \]
so we obtain the upper bound
\[\# \C^\beta(\bQ; \BB) = O_{d,\ep}\left(B^{\frac{2 + \beta}{l} + \ep} \right).\]
and if $\beta \leq 2$, then 
\[0 \leq 2 - \beta < \frac{\log \log B}{\log B}.\]
Therefore, we obtain
\begin{align*} B^{\frac{2\alpha_z (2 - \beta)}{\fd^2}} & \leq B^{\frac{2 \alpha_z( \log \log B/\log B)}{l^2}} \\
& = (\log B)^{\frac{2\alpha_z}{l^2}} \\
& \leq (\log B)^{\frac{1}{l}},   
\end{align*}
as $2\alpha_z \leq l$. This again implies that
\[\# \C^\beta(\bQ; \BB) = O_{d,\ep}\left(B^{\frac{2 + \beta}{l} + \ep} \right).\]
Since $l \geq d$ and $\beta \leq d/k + \log(d \lVert f \rVert)/\log B$, it follows that
\[\#\C^\beta(\bQ; \BB) = O_{d,\ep} \left(B^{\frac{2}{d} + \frac{1}{k} + \ep} \right). \] 
Since $k \geq 2$, it follows that
\[\#\C^\beta(\bQ; \BB) = O_{d,\ep} \left(B^{\frac{2}{d} + \frac{1}{2} + \ep} \right). \] 
Combining these estimates, we obtain
\begin{equation} \label{estimate2} \# X^\beta(\bQ; \BB) \cap \bigcup_{\C \in \Gamma_2} \C =  O_{d, \ep}\left( B^{\psi}B^{\frac{2}{d} + \frac{1}{2} + \ep}\right).\end{equation}
By (\ref{except}), (\ref{gam}), (\ref{estimate1}), and (\ref{estimate2}), we have
\begin{equation} \label{xbetasize} 
\# X^\beta(\bQ;\BB) = O_{d,\ep}\left(B^{2\psi + \ep} + B^{1 + \psi + \ep} + B^{\psi + \frac{2}{d} + \frac{1}{2} + \ep} \right).
\end{equation}
Since we may assume $d \geq 6$ by Greaves \cite{<Gre>}, we obtain
\begin{equation} \label{xbetasize2}
\# X^\beta(\bQ;\BB) = O_{d,\ep}\left(B^{2\psi + \ep} + B^{1 + \psi + \ep}\right).
\end{equation}
It remains to show that if $k/d > 7/18$ and $\beta$ is in the range

\[2 - \frac{\log \log B}{\log B} < \beta \leq \frac{d}{k} + \frac{\log(\lVert f \rVert d)}{\log B},\]
then one can choose $\ep$ so that $\psi < 1$. Let us analyze the expression 
\begin{equation} \label{8E10}\frac{d - k\beta}{d-2k} + \frac{\beta}{2}\end{equation}
as a function of $\beta$. Its derivative is given by
\[\frac{-k}{d-2k} + \frac{1}{2} = \frac{d - 4k}{2(d-2k)},\]
which is negative whenever $k/d > 1/4$. Therefore, by (\ref{7E2}), (\ref{8E10}) viewed as a function of $\beta$, is decreasing. Thus, for any 
\[0 < \eta < \frac{\log \log B}{\log B} \]
with 
\[2 - \eta < \beta \leq 2,\]
we have
\begin{align*}\frac{d - k\beta}{d-2k} + \frac{\beta}{2} & \leq \frac{d - k(2 - \eta)}{d-2k} + \frac{2}{2}\\ 
& = 2 + \frac{k \eta}{d-2k} \\ 
& \leq 2 + \frac{k \log \log B}{(d-2k)\log B}.
\end{align*}
Choose $B$ sufficiently large so that
\begin{equation} \label{8E11} \left(\frac{2(d-2k)}{d}\right)^{1/2}\frac{k \log \log B}{(d-2k)\log B} < \ep.\end{equation}
Let $\lambda = k/d$. Then, by (\ref{8E11}), we have
\[\frac{1}{2} \left(2(1 - 2\lambda \right)^{1/2}\left(1 + \frac{d-2k}{d-2k} + \frac{2}{2}\right) + \ep = \frac{3}{\sqrt{2}} \sqrt{1 - 2\lambda} + \ep.\]
To ensure that $\psi < 1$, we are left to consider the inequality
\[\frac{3 \sqrt{1 - 2 \lambda}}{\sqrt{2}} < 1. \]
This is equivalent to 
\[1 - 2 \lambda < \frac{2}{9},\]
which gives
\[\lambda > \frac{7}{18}.\]
Thus, whenever $k/d > 7/18$ and $\ep$ is sufficiently close to zero, we have $\psi < 1$. This completes the proof of Theorem \ref{MT1}, by virtue of (\ref{7E2}).

%%%%%%%%%%%%%%%%%%%%%%%%%%%%%%%%%%%%%%%%%%%%%%%%%%%%%%%%%
\section{Another proof of Browning's theorem}
\label{S10}
%%%%%%%%%%%%%%%%%%%%%%%%%%%%%%%%%%%%%%%%%%%%%%%%%%%%%%%%%

In this section, we give another proof of Browning's theorem in \cite{<B2>}. It illustrates the differences between our approaches to the determinant method. In \cite{<B2>}, Browning combined elements of the ``affine determinant method" introduced by Heath-Brown in \cite{<HB2>} and Salberger's global determinant method in \cite{<S2>} to prove his result, which is stated below as Theorem \ref{BThm}. Heath-Brown had already shown in \cite{<HB2>} that his affine determinant method could be applied to study integral points on the variety defined by
$$\displaystyle f(x) = yz^k,$$
where $f(x)$ is a polynomial with integral coefficients of degree $d$. More specifically, for irreducible $f(x) \in \bZ[x]$ of degree $d$ with no fixed $k$-th power divisor, Heath-Brown proved that $f$ takes on infinitely many $k$-free values whenever $k \geq (3d+2)/4$. Browning improved on this slightly by showing that Salberger's arguments in \cite{<S2>} can be adopted to augment the affine determinant method to sharpen the above result to $k \geq (3d+1)/4$. \\ \\
We show that our version of the determinant method, detailed in Sections \ref{S2} to \ref{S9}, can also be used to obtain the same result. It is interesting that these two different versions of the determinant method lead to the same conclusion. \\ \\
For convenience, we state Browning's theorem again:
\begin{theorem} (Browning, 2011) \label{BThm} Let $f(x) \in \bZ[x]$ be an irreducible polynomial of degree $d \geq 3$. Suppose that $k \geq (3d+1)/4$. Then, we have
$$\#\{n \in \bZ \cap [1,B]: f(n) \text{ is }k\text{-free}\} \sim c_{f,k} B$$
as $B \rightarrow \infty$, where $c_{f,k}$ is defined as in equation (\ref{E2}). 
\end{theorem}

We first establish some preliminaries analogous to Section \ref{S8}. Recall that we stated, in equation (\ref{E3}), the notation $\displaystyle N_{f,k}(B) = \# \{1 \leq x \leq B: f(x) \text{ is } k\text{-free}\}.$ We define 
$$\displaystyle N(f; b, B) = \#\{1 \leq x \leq B: b^k | f(x)\}.$$
From elementary properties of the Mobius function, we have
$$\displaystyle N_{f,k}(B) = \sum_{b=1}^\infty \mu(b)N(f; b,B).$$
We also have the formula
\[ N(f; b, B) = \rho_f(b^k)\left(\frac{B}{b^k} + O(1)\right),\]
where as we recall from Section \ref{S1}, $\rho_f(m)$ counts the number of congruence classes modulo $m$ for which $f$ vanishes modulo $m$. Browning \cite{<B2>} obtains the estimate
$$\displaystyle \rho_f(b^k) = O( b^\ep)$$
whenever $b$ is square-free, and so we obtain
\[ N(f; b, B) = B \frac{\rho_f(b^k)}{b^k} + O(b^\ep).\] 
We thefore conclude that
\begin{equation} N_{f,k}(B) = B \sum_{b \leq B^{1-\delta}} \frac{\mu(b) \rho_f(b^k)}{b^k} + \sum_{b > B^{1-\delta}} \mu(b) N(f; b, B) + o(B),
\end{equation}
where $\delta$ is a small positive constant. \\ \\
Define the quantity 
$$E(\xi) = \#\{x \in \bZ \cap [1,B] : \exists b > \xi \text{ s.t. } b^k | f(x) \text{ and } \mu^2(b) = 1\}$$ 
for any $\xi \geq 1$. Using the assumption of the theorem that $k > 3d/4 \geq 1$, we find
\begin{equation} N_{f,k}(B) = c_{f,k} B + o(B) + O(E(B^{1-\delta})).
\end{equation}
We now proceed with the proof of Theorem \ref{BThm}.
\begin{proof} The discussion above essentially reduced the proof of Theorem \ref{BThm} to obtaining a satisfactory upper bound for the quantity $E(B^{1-\delta})$. We first homogenize our polynomial $f$ to obtain a binary form $F(x,y)$. As in the proof of Theorem \ref{MT1}, we write $H = B^\beta$, where 
\[1 - \delta < \beta \ll d/k.\] 
We then apply Theorem \ref{MT3} with the weight vector $(1,1,d-k,1)$ and the box $\textbf{B} = (B, 1, H, O(B^d/H^k))$ to the variety defined by 
$$\displaystyle X: F(x,y) - vz^k = 0.$$
Note that this is a weighted projective surface. By Theorem \ref{MT3} we obtain that all points counted by $E(B^{1-\delta})$ lie on an auxiliary curve $\C$ of degree
\[ O_{d,\ep}\left(B^{\frac{1}{2} \left(\frac{d-k}{d}\right)^{1/2} \left(\frac{d-k\beta}{d-k} + \beta\right) + \ep} \right),\]
which assumes its maximum value at $\beta = 1 - \delta$. Then, as per our analysis in the binary form case in Section \ref{S9}, we deduce that we can partition $\C \cap X$ into a collection of  
\[ O_{d,\ep}\left(B^{\frac{1}{2} \left(\frac{d-k}{d}\right)^{1/2} \left(\frac{d-k\beta}{d-k} + \beta\right) + \ep} \right)\]
geometrically irreducible curves $\Gamma$, and an exceptional set $\E$ consisting of 
\[ O_{d,\ep}\left(B^{ \left(\frac{d-k}{d}\right)^{1/2} \left(\frac{d-k\beta}{d-k} + \beta\right) + \ep} \right)\]
points. By \cite{<Hoo>}, we may assume that $d \geq 3$, and as we have shown in Section \ref{S9}, the contribution from each irreducible curve $\D \in \Gamma$ is no more than
\[O_{d,\ep} \left(B^{\frac{1}{3} + \ep} \right), \]
hence it suffices to take $d,k$ to satisfy
\[\left(\frac{d-k}{d}\right)^{1/2} \left(\frac{d-k\beta}{d-k} + \beta\right) < 1\]
for $\beta = 1 - \delta$, with $\delta > 0$ approaching zero. This is satisfied when $k/d > 3/4$, which is equivalent to $k \geq (3d+1)/4$. This completes the proof of Theorem \ref{BThm}. 
\end{proof}

%%%%%%%%%%%%%%%%%%%%%%%%%%%%%%%%%%%%%%%%%%%%%%%%%%%%%%%%%
\section{Proof of Theorem \ref{MT2}}
\label{S11}
%%%%%%%%%%%%%%%%%%%%%%%%%%%%%%%%%%%%%%%%%%%%%%%%%%%%%%%%%

In this section, we give a proof of Theorem \ref{MT2}. Much of the argument remains unchanged from that given in \cite{<ST>}. \\ 

We may assume, as in \cite{<ST>}, that $k \leq D$ since $d$ is at most $D$, and also if an integer is $k$-free it is also $(k+1)$-free. Further, we may assume that the coefficients of $x^D$ and $y^D$ are non-zero, since any binary form $F$ is equivalent under integral unimodular substitutions to a form where the leading $x$ and $y$ coefficients are non-zero. Moreover, unimodular substitutions preserve the discriminant of a binary form. \\

Let $A$ be a positive real number. For any value $0 < \theta \leq 1$ and for any non-zero integer $h$, let us write 
\[ \mathfrak{s}(h) = \prod_{\substack{p \leq A^\theta \\ |h|_p^{-1} \leq A^\theta \\p \nmid \mathfrak{D}}} |h|_p^{-1},\]
where $\mathfrak{D}$ denotes the discriminant of $f$. Write $U$ to be the set of pairs $(a,b) \in \bZ^2$ such that $f(a,b) \ne 0$ and the only primes dividing $\gcd(a,b)$ are those that divide $\mathfrak{D}$. Now, define
$$\displaystyle S(\theta, A) = \prod_{(a,b) \in U} \mathfrak{s}(f(a,b)).$$
One can estimate $S(\theta,A)$ in exactly the same way as in \cite{<ST>} (note that in \cite{<ST>}, they wrote $u$ instead of $A$). In particular, by Section 6 of Stewart-Top \cite{<ST>}, we have the estimate
$$\displaystyle S(\theta, A) \leq A^{5\theta d A^2}.$$
As a consequence, we see that the number of pairs $(a,b) \in U$ such that $|\mathfrak{s}(f(a,b))| \geq A^{1/8}$ is at most $40\theta d A^2$. Now, we may argue as in Lemma 2 of \cite{<EM>} that if $h$ and $b$ are integers such that $|h| \leq A^{1/2}$ and $1 \leq b \leq A$, then there are at most $d$ integers $a$ with $f(a,b) = h$. Hence, the number of pairs of integers $(a,b)$ with $1 \leq a, b \leq A$ and $|f(a,b)| \leq A^{1/2}$ is at most 
$$\displaystyle 3d A^{3/2}.$$
Set $\theta = C_{f,k}/120d$. Define $T$ to be the set of integers $(a,b)$ with $1 \leq a,b \leq A$, $f(a,b)$ is $k$-free, $|f(a,b)| \geq A^{1/2}$, and $\mathfrak{s}(f(a,b)) < A^{1/8}$. By Theorem \ref{MT1} and our choice of $\theta$, we have that there exist constants $C_{22}, C_{23} > 0$, which depend on $f$ and $k$, such that whenever $A > C_{22}$, we have
\begin{equation}\displaystyle \# T > \frac{1}{2} C_{23}A^2. \end{equation}
We invoke the work of Stewart in \cite{<Stew>} on estimating the number of solutions to Thue equations. Recall that for any integer $h$, $\omega(h)$ denotes the number of distinct prime factors of $h$. Let $h$ be an integer for which there exists $(a,b) \in T$ such that
\begin{equation} \label{10E2} \displaystyle f(a,b) = h. \end{equation}
Write $h = \mathfrak{s}(f(a,b)) \cdot g$. Since by assumption we have $\mathfrak{s}(f(a,b)) \leq A^{1/8}$ and $|f(a,b)| \geq A^{1/2}$, it follows that $|\mathfrak{s}(f(a,b))| \leq |h|^{1/4}$ and consequently, $|g| \geq |h|^{3/4}$. If $A$ is chosen to be greater than $|\mathfrak{D}|^{24}$ and $|h| \geq |\mathfrak{D}|^{12}$, then choosing $\ep = 1/12$ and applying Corollary 1 of \cite{<Stew>} we obtain that the number of solutions to equation (\ref{10E2}) is at most
$$\displaystyle 5600 d^{1 + \omega(g)}.$$
Observe that trivially we have the bound
\begin{equation} \label{10E3} \displaystyle |f(a,b)| \leq d \lVert f \rVert A^d. \end{equation}
Note that by construction, the prime divisors of $g$ either divide $\mathfrak{D}$ or satisfy $|f(a,b)|_p^{-1} \geq A^\theta$. Hence, by choosing $A$ so that $A^\theta \geq d \lVert f \rVert$, we have
$$\displaystyle \omega(g) \leq \omega(\mathfrak{D}) + (d+1)/\theta.$$
The second term on the right hand side in the above equation is from the worst case, where each prime $p$ such that $|f(a,b)|_p^{-1} \geq A^\theta$ divides $f(a,b)$ with multiplicity one. If there are more than $(d+1)/\theta$ of such primes, then we will have $|f(a,b)| \geq A^{\theta \cdot (d+1)/\theta} = A^{d+1}$, which yields a contradiction to equation (\ref{10E3}) as we chose $A \geq A^\theta \geq d \lVert f \rVert$. Hence, there exist constants $C_{24}, C_{25}$ such that if $A > C_{24}$, then the number of distinct pairs $(a,b) \in T$ is at least $C_{25} A^2$. \\ \\
To finish the proof of the theorem, let $B$ be a real number with $B > d \lVert f \rVert C_{24}^d$ and write $A = (B/d \lVert f \rVert)^{1/d}$. Note that $A > C_{24}$. With this choice of $A$, we have that whenever $(a,b) \in T$, we have $|f(a,b)| \leq B$. Hence,
$$\displaystyle R_k(B) \geq \# T \geq C_{25} (B/d \lVert f \rVert)^{2/d},$$
which completes the proof of Theorem \ref{MT2}.


\newpage

\begin{thebibliography}{10}

\bibitem{BR}
M.~Beltrametti, L.~Robbiano, \emph{Introduction to the theory of weighted projective spaces}, Expo. Math, \textbf{4} (1986), 111-162.

\bibitem{<BP>}
E.~Bombieri, J.~Pila, \emph{The number of integral points on arcs and ovals}, Duke Mathematical Journal, (2) \textbf{59} (1989), 337-357.

\bibitem{<BV>}
E.~Bombieri, J.~Vaaler, \emph{On Siegel's lemma}, Inventiones Mathematicae, \textbf{73} (1983), 11-32.

\bibitem{Bro}
N.~Broberg, \emph{A note on a paper by R.~Heath-Brown: ``The density of rational points on curves and surfaces"}, J.~reine angew.~Math. \textbf{571} (2004), 159-178.

\bibitem{<Bro0>}
N.~Broberg, \emph{Rational points on weighted plane curves}, preprint. 

\bibitem{B1}
T.~D.~Browning, \emph{Quantitative Arithmetic of Projective Varieties}, Progress in Mathematics, \textbf{277} (2009).

\bibitem{<B2>}
T.~D.~Browning, \emph{Power-free values of polynomials}, Arch.~Math. (2) \textbf{96} (2011), 139-150. 

\bibitem{<BHS>}
T.~D.~Browning, D.~R.~Heath-Brown, P.~Salberger, \emph{Counting rational points on algebraic varieties}, Duke Mathematical Journal, (3) \textbf{132} (2006), 545-578.

\bibitem{CLO0}
D.~A.~Cox, J.~B.~Little, D.~O'Shea, \emph{Using Algebraic Geometry}, Revised Second Edition (2005), Springer-Verlag.

\bibitem{<CLO>}
D.~A.~Cox, J.~B.~Little, D.~O'Shea, \emph{Ideals, Varieties, and Algorithms}, Third Edition (2007), Springer-Verlag.

\bibitem{<CLS>}
D.~A.~Cox, J.~B.~Little, H.~K.~Schenck \emph{Toric Varieties}, Graduate Studies in Mathematics, \textbf{124} (2011), American Mathematical Society.

\bibitem{<Dol>}
I.~Dolgachev, \emph{Weighted projective varieties}, Group Actions and Vector Fields (1982), Springer.

\bibitem{<EH>}
D.~Eisenbud, J.~Harris, \emph{The Geometry of Schemes}, (2000), Springer-Verlag.

\bibitem{<E>}
P.~Erd\H{o}s, \emph{Arithmetical properties of polynomials}, J. London Math. Soc. \textbf{28} (1953), 416-425.

\bibitem{<EM>}
P.~Erd\H{o}s, K.~Mahler, \emph{On the number of integers which can be represented by a binary form}, J.~London~Math.~Soc, \textbf{13} (1938), 134-139.

\bibitem{<Fil>}
M.~Filaseta, \emph{Powerfree values of binary forms}, Journal of Number Theory \textbf{49} (1994), 250-268.

\bibitem{Fu}
W.~Fulton, \emph{Intersection Theory}, Springer-Verlag 1984.

\bibitem{<GM>}
F.~Q.~Gouv\^{e}a, B. Mazur, \emph{The square-free sieve and the rank of elliptic curves}, Journal of the American Mathematical Society, (1) \textbf{4} (1991), 793-805.

\bibitem{<Gran>}
A.~Granville, \emph{$ABC$ allows us to count squarefrees}, International Mathematics Research Notices, \textbf{9} (1998).

\bibitem{<Gre>}
G.~Greaves, \emph{Power-free values of binary forms}, Q.~J.~Math, (2) \textbf{43} (1992), 45-65.

\bibitem{<Har>}
R.~Hartshorne, \emph{Algebraic Geometry}, Graduate Texts in Mathematics \textbf{52} (1977), Springer-Verlag.

\bibitem{<Har1>}
R.~Hartshorne, \emph{Deformation Theory}, Graduate Texts in Mathematics \textbf{257} (2010), Springer-Verlag.

\bibitem{<HB1>}
D.~R.~Heath-Brown, \emph{The density of rational points on curves and surfaces}, The Annals of Mathematics (2) \textbf{155} (2002), 553-598.

\bibitem{<HB2>}
D.~R.~Heath-Brown, \emph{Counting rational points on algebraic varieties}, Analytic number theory, 51–95, Lecture Notes in Math., 1891, Springer, Berlin, 2006. 

\bibitem{<HB3>}
D.~R.~Heath-Brown, \emph{Sums and differences of three $k$-th powers}, Journal of Number Theory, \textbf{129} (2009), 1579-1594.

\bibitem{<HB4>}
D.~R.~Heath-Brown, \emph{Powerfree values of polynomials}, Q.~J.~Math, (2) \textbf{64} (2013), 177-188.

\bibitem{<Hel>}
H.~A.~Helfgott, \emph{On the square-free sieve}, Acta Arithmetica, \textbf{115} (2004), 349-402.

\bibitem{<HS>}
M.~Hindry, J.~Silverman, \emph{Diophantine Geometry: An Introduction}, (2000), Springer-Verlag.

\bibitem{<Hoo>}
C.~Hooley, \emph{On the power free values of polynomials}, Mathematika \textbf{14} (1967), 21-26.

\bibitem{<Hoo2>}
C.~Hooley, \emph{On the power-free values of polynomials in two variables}, Analytic number theory, 235-266, Camb.~Univ.~Press, 2009.

\bibitem{<Hoo3>}
C.~Hooley, \emph{On the power-free values of polynomials in two variables: II}, Journal of Number Theory, \textbf{129} (2009), 1443-1455.

\bibitem{<LW>}
S.~Lang, A.~Weil, \emph{Number of points of varieties over finite fields}, American Journal of Mathematics, (4) \textbf{76} (1954), 819-827.

\bibitem{<Mac>}
D.~MacLagan, \emph{Notes on Hilbert schemes}, \url{http://homepages.warwick.ac.uk/staff/D.Maclagan/papers/HilbertSchemesNotes.pdf}.

\bibitem{<Mum>}
D.~Mumford, \emph{The Red Book of Varieties and Schemes}, Lect. Notes Math., \textbf{1358}, Springer-Verlag, 1988.

\bibitem{<MP>}
R.~Murty, H.~Pasten, \emph{Counting square free values of polynomials with error term}, International Journal of Number Theory, (7) \textbf{10} (2014), 1743-1760.

\bibitem{<Nai1>}
M.~Nair, \emph{Power free values of polynomials}, Mathematika \textbf{23} (1976), 159-183.

\bibitem{<Nai2>}
M.~Nair, \emph{Power free values of polynomials II}, Proceedings of the London Mathematical Society (3) \textbf{38} (1979), 353-368.

\bibitem{<Poo>}
B.~Poonen, \emph{Squarefree values of multivariate polynomials}, Duke Math. J., (2) \textbf{118} (2003), 353-373.

\bibitem{<Re>}
M.~Reid, \emph{Graded rings and varieties in weighted projective space}, \url{http://homepages.warwick.ac.uk/~masda/surf/more/grad.pdf}.

\bibitem{<Ric>}
G.~Ricci, \emph{Ricerche aritmetiche sui polinomi}, Rend. Circ. Mat. Palermo \textbf{57} (1933), 433-475.

\bibitem{<RS>}
J.~B.~Rosser, L.~Schoenfeld, \emph{Approximate formulas for some functions of prime numbers}, Illinois J.~Math., (1) \textbf{6} (1962), 64-94. 

\bibitem{<S1>}
P.~Salberger, \emph{On the density of rational and integral points on algebraic varieties}, J.~reine angew.~Math. \textbf{606} (2007), 123-147.

\bibitem{<S2>}
P.~Salberger, \emph{Counting rational points on projective varieties}, Preprint 2009.

\bibitem{<Stan>}
R.~Stanley, \emph{Some restricted weighted sums}. MathOverflow, 2012. URL (accessed on 2014-05-22):
\url{http://mathoverflow.net/questions/90381/some-restricted-weighted-sums}.

\bibitem{<Stew>}
C.~L.~Stewart, \emph{On the number of solutions of polynomial congruences and Thue equations}, Journal of the American Mathematical Society, (4) \textbf{4} (1991), 793-835.

\bibitem{<ST>}
C.~L.~Stewart, J.~Top, \emph{On ranks of twists of elliptic curves and power-free values of binary forms}, Journal of the American Mathematical Society, (4) \textbf{8} (1995), 943-972.

\bibitem{<Wal>}
M.~Walsh, \emph{Bounded rational points on curves}, International Mathematics Research Notices (14) \textbf{2015} (2015), 5644-5658.






















\end{thebibliography}
\end{document}